\documentclass[10pt]{amsart}
\usepackage{amssymb,amsmath}

\newtheorem{theorem}{Theorem}[section]
\newtheorem{proposition}[theorem]{Proposition}
\newtheorem{lemma}[theorem]{Lemma}

\newtheorem*{theorem*}{Theorem}	
\theoremstyle{definition}
\newtheorem{definition}[theorem]{Definition}
\newtheorem{remark}[theorem]{Remark}
\newtheorem{example}[theorem]{Example}

\begin{document}	

\title[Conditional nonlinear expectations]{Conditional nonlinear expectations}
\author[Daniel Bartl]{Daniel Bartl$^*$}	
\thanks{$^*$Department of Mathematics, University of Vienna, daniel.bartl@unvie.ac.at}
\date{\today}
%45Nxx - Abstract integral equations, integral equations in abstract spaces
%93E20 - Optimal stochastic control
%46N10 - Applications in optimization, convex analysis, mathematical programming, economics

\begin{abstract} 
	Let $\Omega$ be a Polish space with Borel $\sigma$-field $\mathcal{F}$ and countably generated sub $\sigma$-field $\mathcal{G}\subset\mathcal{F}$. Denote by $\mathcal{L}(\mathcal{F})$ the set of all bounded $\mathcal{F}$-upper semianalytic functions from $\Omega$ to the reals and by $\mathcal{L}(\mathcal{G})$ the subset of $\mathcal{G}$-upper semianalytic functions. Let $\mathcal{E}(\cdot|\mathcal{G})\colon\mathcal{L}(\mathcal{F})\to\mathcal{L}(\mathcal{G})$ be a sublinear increasing functional which leaves $\mathcal{L}(\mathcal{G})$ invariant. It is shown that there exists a $\mathcal{G}$-analytic set-valued mapping $\mathcal{P}_{\mathcal{G}}$ from $\Omega$ to the set of probabilities which are concentrated on atoms of $\mathcal{G}$ with compact convex values such that $\mathcal{E}(X|\mathcal{G})(\omega)=$ $\sup_{P\in\mathcal{P}_{\mathcal{G}}(\omega)} E_P[X]$ if and only if $\mathcal{E}(\cdot |\mathcal{G})$ is pointwise continuous from below and continuous from above on the continuous functions. Further, given another sublinear increasing functional $\mathcal{E}(\cdot)\colon\mathcal{L}(\mathcal{F})\to\mathbb{R}$ which leaves the constants invariant, the tower property $\mathcal{E}(\cdot)=\mathcal{E}(\mathcal{E}(\cdot|\mathcal{G}))$ is characterized via a pasting property of the representing sets of probabilities, and the importance of analytic functions is explained. Finally, it is characterized when a nonlinear version of Fubini's theorem holds true and when the product of a set of probabilities and a set of kernels is compact.
\end{abstract}

\maketitle
\setcounter{equation}{0}

\section{Introduction}

The Daniell-Stone theorem is a basic but essential result in measure- and integration theory and states that an linear increasing functional preserving the constants has an integral representation w.r.t.~a probability measure if and only if it satisfies the monotone convergence property.
Its nonlinear version is significantly more involved, builds on capacity theory, and is mainly due to Choquet \cite{choquet1959forme}:
Let $\Omega$ be a Polish space with Borel $\sigma$-field $\mathcal{F}$ and denote by $\mathfrak{P}(\Omega)$ the set of all probabilities on $\mathcal{F}$.
Further write $C_b(\Omega)$ and $\mathcal{L}(\mathcal{F})$ for the set of bounded functions from $\Omega$ to $\mathbb{R}$ which are continuous and $\mathcal{F}$-upper semianalytic, respectively (see Appendix \ref{sec:appendix.analytic}). 

\begin{theorem*}[Choquet]
	Let $\mathcal{E}(\cdot)\colon\mathcal{L}(\mathcal{F})\to\mathbb{R}$ be a  sublinear expectation (i.e.~a sublinear increasing functional satisfying  $\mathcal{E}(X)=X$ for all constant functions $X\in\mathbb{R}$).
	Then there exists a convex and weakly compact set $\mathcal{P}\subset\mathfrak{P}(\Omega)$ such that
	\begin{align}
	\label{eq:intro.choquet}
	\mathcal{E}(X)=\sup_{P\in\mathcal{P}} E_P[X]\quad\text{for all }X\in\mathcal{L}(\mathcal{F})
	\end{align}
	if and only if $\mathcal{E}(X_n)\downarrow\mathcal{E}(X)$ for every sequence $X_n\in C_b(\Omega)$ with $X_n\downarrow X\in \mathcal{L}(\mathcal{F})$ pointwise and $\mathcal{E}(X_n)\uparrow \mathcal{E}(X)$ for every sequence $X_n\in\mathcal{L}(\mathcal{F})$ with $X_n\uparrow X\in\mathcal{L}(\mathcal{F})$ pointwise.
\end{theorem*}	

See Choquet's original work \cite{choquet1959forme} for the theorem in a different form and e.g.~\cite{beiglbock2015complete,kellerer1984duality} as well as  \cite[Section 2]{bartl2017robust} for applications and the statement in precisely this form. 
For convenience, a sketch of the proof is given in Appendix \ref{sec:app.choquet}.

\vspace{0.5em}

The first goal of this article is to obtain a result of this type for conditional sublinear expectations.
To gain some feeling and find the right formulation, it might be helpful to construct those objects first.
To that end, let $\mathcal{G}\subset\mathcal{F}$ be a countably generated sub $\sigma$-field.
Given a set of probabilities $\mathcal{P}$, the first approach ``$\sup_{P\in\mathcal{P}} E_P[X|\mathcal{G}]$'' is (in general) not well-defined as $E_P[\cdot|\mathcal{G}]$ is defined only up to $P$-zero sets which are usually not the same over the class $\mathcal{P}$, and the uncountable supremum over $P\in\mathcal{P}$ may fail to be measurable.
Instead, a representation over kernels seems more feasible: 
denote by $P_\mathcal{G}$ the regular version of the conditional probability so that $E_P[X|P]=E_{P_\mathcal{G}}[X]$ almost surely and $P_\mathcal{G}(\omega)([\omega]_\mathcal{G})=1$ for almost all $\omega$, where $[\omega]_\mathcal{G}:=\bigcap \{ A : \omega\in A\in\mathcal{G}\}$ (note that if $\mathcal{G}=\sigma(\phi)$ for some Borel $\phi \colon\Omega\to S$, where $S$ is another Polish space, then $\mathcal{G}$ is countably generated and $[\omega]_\mathcal{G}=\{\eta\in\Omega : \phi(\eta)=\phi(\omega)\}$).
Then, if for each $\omega\in\Omega$ one is given a set of probabilities $\mathcal{P}_\mathcal{G}(\omega)$ on $\Omega$ such that $P([\omega]_\mathcal{G}=1)$ for all $P\in\mathcal{P}_\mathcal{G}(\omega)$, one can define 
\begin{align} 
\label{eq:intro.def.cond}
\mathcal{E}(\cdot|\mathcal{G}) \colon \mathcal{L}(\Omega) \to \mathbb{R}^\Omega,
\quad \mathcal{E}(X|\mathcal{G})(\omega) := \sup_{P\in\mathcal{P}_\mathcal{G}(\omega)}  E_P[X].
\end{align} 
The functional $\mathcal{E}(\cdot|\mathcal{G})$ is pointwise sublinear and increasing, and $\mathcal{E}(X|\mathcal{G})=X$ for every $X\in\mathcal{L}(\mathcal{G})$, where $\mathcal{L}(\mathcal{G})$ is the set of all bounded $\mathcal{G}$-upper semianalytic functions.
Further, under the assumption that 
\[\mathop{\mathrm{Graph}}(\mathcal{P}_\mathcal{G}):=\{ (\omega,P)\in \Omega\times\mathfrak{P}(\Omega) : P\in\mathcal{P}_\mathcal{G}(\omega)\} \text{ is an } \mathcal{G}\otimes\mathcal{B}(\mathfrak{P}(\Omega))\text{-analytic set,}\] the powerful theory of Luzin and Suslin applies and guarantees that $\mathcal{E}(\cdot|\mathcal{G})$ is in fact a mapping from $\mathcal{L}(\mathcal{F})$ to $\mathcal{L}(\mathcal{G})$.
This pointwise construction of conditional sublinear expectations (i.e.~sublinear increasing functionals $\mathcal{E}(\cdot|\mathcal{G})\colon\mathcal{L}(\mathcal{F})\to\mathcal{L}(\mathcal{G})$ satisfying $\mathcal{E}(X|\mathcal{G})=X$ for all $X\in\mathcal{L}(\mathcal{G})$) has been used several times over the last decades, see for instance \cite{bertsekas1978stochastic,bouchard2015arbitrage,dellacherie2011probabilities,karoui2013capacities,nutz2013constructing}.
The first main result of this article is the reverse question, that is, whether an arbitrary mapping $\mathcal{E}(\cdot|\mathcal{G})\colon\mathcal{L}(\mathcal{F})\to\mathcal{L}(\mathcal{G})$ has an associated family of probabilities $(\mathcal{P}_\mathcal{G}(\omega))_\omega$ such that \eqref{eq:intro.def.cond} holds, and if $\mathop{\mathrm{Graph}}(\mathcal{P}_\mathcal{G})$ needs to be analytic, which was not studied so far.

\begin{theorem}
\label{thm:cond.rep.sublin}
	Let $\mathcal{E}(\cdot|\mathcal{G})\colon \mathcal{L}(\mathcal{F})\to \mathcal{L}(\mathcal{G})$	be a conditional nonlinear expectation.
	Then there exists a set-valued mapping $\mathcal{P}_\mathcal{G}\colon\Omega\to\mathfrak{P}(\Omega)$ with nonempty, convex, and weakly compact values such that $\mathop{\mathrm{Graph}}(\mathcal{P}_\mathcal{G})$ is an $\mathcal{G}\otimes\mathcal{B}(\mathfrak{P}(\Omega))$-analytic set and $P\in\mathcal{P}_\mathcal{G}(\omega)$ implies $P([\omega]_\mathcal{G})=1$	for which
	\begin{align}
	\label{eq:intro.choquet.cond}
	\mathcal{E}(X|\mathcal{G})(\omega)=\sup_{P\in\mathcal{P}_\mathcal{G}(\omega)}  E_P[X]
	\quad\text{for every }\omega\in\Omega\text{ and }X\in\mathcal{L}(\mathcal{F}) 
	\end{align}
	if and only if 
	$\mathcal{E}(X_n|\mathcal{G})\downarrow\mathcal{E}(X|\mathcal{G})$ pointwise for every sequence $X_n\in C_b(\Omega)$ with 
	$X_n\downarrow X\in \mathcal{L}(\mathcal{F})$ pointwise and
	$\mathcal{E}(X_n|\mathcal{G})\uparrow \mathcal{E}(X|\mathcal{G})$ pointwise for every sequence 
	$X_n\in\mathcal{L}(\mathcal{F})$ with $X_n\uparrow X\in\mathcal{L}(\mathcal{F})$ pointwise.
\end{theorem}

In case of linear conditional expectations, next to monotonicity and preservation of $\mathcal{G}$-measurable functions, the most important property is the tower property $E[\cdot]=E[E[\cdot|\mathcal{G}]]$.
While the first two properties are part of the definition of conditional sublinear expectations, the tower property $\mathcal{E}(\cdot)=\mathcal{E}(\mathcal{E}(\cdot|\mathcal{G}))$  (also called dynamic programming principle) does not hold in general.
However, it is possible to characterize when it does hold true on the level of representing probabilities.
A more general version of the proceeding theorem is formulated in Theorem \ref{thm:char.tower}.

\begin{theorem}
\label{thm:char.tower.sublin}
	Assume that $\mathcal{E}(\cdot)$ and $\mathcal{E}(\cdot |\mathcal{G})$ satisfy the assumptions of the previous theorems and therefore have the representations \eqref{eq:intro.choquet} and \eqref{eq:intro.choquet.cond}, respectively. 
	If further $\mathcal{E}(X|\mathcal{G})$ is $\mathcal{G}$-measurable for every $X\in C_b(\Omega)$ and $\mathcal{G}=\sigma(\phi)$ for some continuous $\phi\colon\Omega\to S$ where $S$ is another Polish space, then
	$\mathcal{E}(\cdot)=\mathcal{E}(\mathcal{E}(\cdot|\mathcal{G}))$ if and only if
	\[ \mathcal{P}
	=\mathcal{P}\otimes \mathcal{P}_{\mathcal{G}}
	:=\{ Q\otimes R :Q\in\mathcal{P} \text{ and } R(\cdot)\in \mathcal{P}_{\mathcal{G}}(\cdot) 
	\,\,Q\text{-almost surly}\},\]
	where $R\colon \Omega\to \mathfrak{P}(\Omega)$ is $\mathcal{G}$-(universally) measurable.
\end{theorem}

Here $Q\otimes R\in\mathfrak{P}(\Omega)$ is defined by $E_{Q\otimes R}[X]=E_{Q(d\omega)}[E_{R(\omega)}[X]]$ for $X\in\mathcal{L}(\mathcal{F})$.
Note that it follows rather directly from results on analytic sets
that $\mathcal{E}(\cdot)=\mathcal{E}(\mathcal{E}(\cdot|\mathcal{G}))$  whenever $\mathcal{P}=\mathcal{P}\otimes \mathcal{P}_\mathcal{G}$.
The actual statement of the theorem is the reverse implication, which has the following application.

\begin{theorem}
\label{thm:compactness.of.product}
	Assume that $\mathcal{G}=\sigma(\phi)$ for some continuous $\phi\colon\Omega\to S$ where $S$ is another Polish space.
	Let $\mathcal{P}_\mathcal{G}\colon \Omega\rightsquigarrow\mathfrak{P}(\Omega)$ be a set-valued mapping with convex values such that $P([\omega]_\mathcal{G})=1$ for $P\in\mathcal{P}_\mathcal{G}(\omega)$ and $\omega\in\Omega$, and $\omega\mapsto \sup_{P\in\mathcal{P}_\mathcal{G}(\omega)} E_P[X]$ is $\mathcal{G}$-measurable for every $X\in C_b(\Omega)$.
	Then
	\[\mathcal{P}\otimes\mathcal{P}_\mathcal{G}\subset\mathfrak{P}(\Omega)
	\text{ is compact for every compact convex set }\mathcal{P}\subset\mathfrak{P}(\Omega)\]
	if and only if 
	$\mathcal{P}_\mathcal{G}$ has compact values and $\omega\mapsto \max_{P\in\mathcal{P}_\mathcal{G}(\omega)} E_P[X]$ is upper semicontinuous for every $X\in C_b(\Omega)$.
\end{theorem}

An application of this theorem to robust mathematical finance is given in Section \ref{sec:compactness.kernel}.
It is further characterized in which cases a nonlinear version of Fubini's theorem  (on the interchanging of the order of integration) holds true in Proposition \ref{prop:fubini}.
Examples to dynamic risk measures in the presence of Knightian uncertainty in discrete time as well as the controlled Brownian motion are studied in Example \ref{ex:avar.robust} and Example \ref{ex:G.brownian}, respectively.

Finally, notice that the choice of upper semianalytic functions as opposed to canonical ones as Borel- or universally measurable functions is important (in fact, fundamental):
As Theorem \ref{thm:cond.rep.sublin} should be an extension of Choquet's theorem,  it should at least include the case where $\mathcal{P}_\mathcal{G}(\omega)$ in \eqref{eq:intro.choquet.cond} depends as little as possible on $\omega\in\Omega$.	
However, already then one can not expect $\mathcal{E}(X|\mathcal{G})$ to be $\mathcal{G}$-measurable for $\mathcal{F}$-measurable $X$.
As a matter of fact it is not possible (in general) to work with vector spaces containing all bounded $\mathcal{F}$ or $\mathcal{G}$-measurable functions.
This is discussed in detail in Remark \ref{rem:why.usa}.

\vspace{0.5em}

The systematic treatment of nonlinear expectations with domain $C_b(\Omega)$ (and its completion with respect to a seminorm induced by a capacity) started in \cite{denis2011function,peng2010nonlinear}, in particular with connections to the nonlinear Brownian motion introduced by Peng \cite{peng2008multi}.
The extension of the latter to measurable (upper semianalytic) functions was carried out in \cite{nutz2013constructing} and generalized to nonlinear L{\'e}vy-processes in \cite{neufeld2017nonlinear}.
In the framework of robust mathematical finance (in discrete time), conditional nonlinear expectations were introduced in the seminal work of Bouchard and Nutz \cite{bouchard2015arbitrage} and successfully applied several times, see Section \ref{sec:sub.risk} for references. 
Further applications can be found e.g~in the context of (non-exponential) large deviations \cite{eckstein2017extended,lacker2016non} where the tower property (therein refereed to as tensorization) plays a crucial role, or in the context of fully nonlinear PDE's where the tower property is the flow/semigroup  property, see \cite{denk2017semigroup} and references therein.
If $\mathcal{L}(\mathcal{F})$ is replaced by the quotient space $L^\infty(\Omega,P^\ast)$ with respect to some reference measure $P^\ast$, nonlinear expectations and their dual representation were already studied in detail
due to the relation to risk measures, see e.g.~\cite{delbaen2002coherent,follmer2011stochastic}.
In a similar manner, \cite{cohen2012quasi} works in a setting where essential suprema are assumed to exist; see \cite{maggis2016fatou} for the characterization of this assumption.

\vspace{0.5em}

The rest of this article is organized as follows: 
All results and proofs for conditional nonlinear expectations (in a more general convex instead of sublinear form) are presented in Section \ref{sec:main.results}.
Applications, the proof of Theorem \ref{thm:compactness.of.product}, and examples are given in Section \ref{sec:example.extension}.
A summary together with basic facts about analytic sets is given in Appendix \ref{sec:appendix.analytic}, and Appendix \ref{sec:app.choquet} contains a sketch of Choquet's theorem as stated in the introduction.

\section{Main results}
\label{sec:main.results}

First fix notation.
Recall that $\Omega$ is a Polish space with Borel $\sigma$-field $\mathcal{F}$ and countably generated sub $\sigma$-field $\mathcal{G}$.
For any $\sigma$-field $\mathcal{H}$ on $\Omega$, the set $\mathcal{L}(\mathcal{H})$ denotes the set of all bounded $\mathcal{H}$-upper semianalytic functions from $\Omega$ to $\mathbb{R}$.
A short summary of analytic (and universally measurable) sets and functions is given in Appendix \ref{sec:appendix.analytic}.
Write $C_b(\Omega)$ and $usc_b(\Omega)$ for the set of bounded functions from $\Omega$ to $\mathbb{R}$ which are continuous and upper semicontinuous, respectively.
For any function $X\colon\Omega\to[-\infty,+\infty]$, define it's maximum norm $\|X\|_\infty:=\sup_{\omega\in\Omega} |X(\omega)|$.
The set of $\sigma$-additive Borel probability measures on $\Omega$ is denoted by $\mathfrak{P}(\Omega)$ and endowed with the weak topology $\sigma(\mathfrak{P}(\Omega),C_b(\Omega))$, i.e.~the coarsest topology making the mappings $P\mapsto E_P[X]$ continuous for every $X\in C_b(\Omega)$.
This renders $\mathfrak{P}(\Omega)$ a Polish space.
For $P\in\mathfrak{P}(\Omega)$, denote by $P_\mathcal{G}$ the regular version of the conditional probability.
For $P\in\mathfrak{P}(\Omega)$ and $P$-measurable $R\colon\Omega\to\mathfrak{P}(\Omega)$ write $Q:=P\otimes R$ for the probability $Q\in\mathfrak{P}(\Omega)$ defined by $E_Q[X]=E_{P(d\omega)}[E_{R(\omega)}[X]]$; in particular $P=P\otimes P_\mathcal{G}$.
For any Polish space $S$, denote by $\mathcal{B}(S)$ the Borel $\sigma$-field.
Product spaces are endowed with the  product topology
and when functions are in consideration, (in-) equalities and convergence
is to be understood in a pointwise sense, unless stated otherwise.

\begin{definition}
	\label{def:expectation}
	A mapping $\mathcal{E}(\cdot|\mathcal{G})\colon \mathcal{L}(\mathcal{F})\to\mathcal{L}(\mathcal{G})$ is called conditional nonlinear expectation, if for all $X,Y\in\mathcal{L}(\mathcal{F})$ one has
	\begin{itemize}
	\addtolength{\itemindent}{-2em}
	\item $\mathcal{E}(X|\mathcal{G})\leq \mathcal{E}(Y|\mathcal{G})$ whenever $X\leq Y$,
	\item $\mathcal{E}(X|\mathcal{G})=X$ whenever $X\in\mathcal{L}(\mathcal{G})$,
	\item $\mathcal{E}(\lambda X+(1-\lambda)Y|\mathcal{G})\leq \lambda\mathcal{E}(X|\mathcal{G})+(1-\lambda)\mathcal{E}(Y|\mathcal{G})$
	for all $\lambda\in[0,1]$.
	\end{itemize}
	Further $\mathcal{E}(\cdot|\mathcal{G})$ is said to be a conditional sublinear expectation if in addition
	\begin{itemize}
	\addtolength{\itemindent}{-2em}
	\item $\mathcal{E}(\lambda X|\mathcal{G})=\lambda\mathcal{E}(X|\mathcal{G})$ 
	for all $\lambda\in[0,+\infty)$.
	\end{itemize}
\end{definition}

\subsection{Continuity and representation}

The goal of this section is to establish a conditional version of Choquet's theorem for nonlinear expectations. 
The sublinear case, stated in Theorem \ref{thm:cond.rep.sublin}, will be a special case; its proof is given at the end of this section.

\begin{definition}
	A conditional nonlinear expectation $\mathcal{E}(\cdot|\mathcal{G})$ is said to be continuous from above (on $C_b(\Omega)$) if
	\begin{itemize}
	\item[(A)]
	$\mathcal{E}(X_n|\mathcal{G})\downarrow\mathcal{E}(X|\mathcal{G})$ for all sequences $X_n\in C_b(\Omega)$ with $X_n\downarrow X\in \mathcal{L}(\mathcal{F})$.
	\end{itemize}
	Similarly, $\mathcal{E}(\cdot|\mathcal{G})$ is said to be continuous from below (on $\mathcal{L}(\Omega)$) if
	\begin{itemize}
	\item[(B)] $\mathcal{E}(X_n|\mathcal{G})\uparrow \mathcal{E}(X|\mathcal{G})$ for all sequences $X_n\in\mathcal{L}(\mathcal{F})$	with $X_n\uparrow X\in\mathcal{L}(\mathcal{F})$.
	\end{itemize}
\end{definition}

Before stating the main result, some facts about the continuity properties are stated which are worth mentioning.

\begin{remark}
\label{rem:cont.above.usc}
	If $\mathcal{E}(\cdot|\mathcal{G})$ is a conditional nonlinear expectation which satisfies (A) and (B), then $\mathcal{E}(X_n|\mathcal{G})\downarrow \mathcal{E}(X|\mathcal{G})$ for every sequence $X_n\in usc_b(\Omega)$ with $X_n\downarrow X\in\mathcal{L}(\mathcal{F})$.
	This is shown in the proof of Theorem \ref{thm:cond.rep}.
\end{remark}

\begin{remark}
	\label{rem:dual.char.A}
	A conditional nonlinear expectation $\mathcal{E}(\cdot|\mathcal{G})$ satisfies (A) if and only if it satisfies both
	\begin{itemize}
	\item[(A')] $\mathcal{E}(X_n|\mathcal{G})\downarrow 0$ for all sequences $X_n\in C_b(\Omega)$ with $X_n\downarrow 0$,
	\item[(A'')] $\sup_{X\in C_b(\Omega)} (E_P[X] -\mathcal{E}(X|\mathcal{G}))=\sup_{X\in usc_b(\Omega)} (E_P[X] -\mathcal{E}(X|\mathcal{G}))$
	for every probability $P\in\mathfrak{P}(\Omega)$.
	\end{itemize} 
\end{remark}
\begin{proof}
	Since $\Omega$ is a Polish space, every upper semicontinuous function can be written as the decreasing limit of a sequence of continuous functions. 
	Therefore it is clear that (A) implies (A') and (A'').
	The other direction will be shown within the proof of Theorem \ref{thm:cond.rep}. 
\end{proof}

\begin{remark}
	\label{rem:tightness}
	In the non-conditional case (i.e.~when $\mathcal{G}$ is trivial) condition (A') is equivalent to the well-known ``tightness" condition: 
	There exists a sequence of compact sets $K_n\subset\Omega$ such that $\mathcal{E}(m 1_{K_n^c})\downarrow0$ for every positive real number $m$.
	In the conditional case this no longer holds true.
\end{remark}
\begin{proof}
	Here it is only shown that in general there is no sequence of compact sets $K_n\subset\Omega$ such that $\mathcal{E}(1_{K_n^c}|\mathcal{G})\downarrow0$.
	The equivalence of (A') and the tightness condition in the non-conditional case is shown within the proof of Theorem \ref{thm:cond.rep}.
	Let $\Omega=\mathbb{R}^\mathbb{N}\times\mathbb{R}^\mathbb{N}$ with canonical elements $\omega=(\omega_1,\omega_2)$ and $\mathcal{G}$ generated by the first coordinate, that is, $\mathcal{G}=\sigma(\omega\mapsto \omega_1)$.
	Define 
	\[\mathcal{E}(X|\mathcal{G})(\omega):=X(\omega_1,\omega_1) \quad\text{for }\omega=(\omega_1,\omega_2)\in \Omega
	\text{ and } X\in\mathcal{L}(\mathcal{F}).\]
	By \cite[Lemma 7.30]{bertsekas1978stochastic}, one has $\mathcal{E}(X|\mathcal{G})\in\mathcal{L}(\mathcal{G})$ for every $X\in\mathcal{L}(\mathcal{F})$.
	Therefore $\mathcal{E}(\cdot|\mathcal{G})$ is a conditional sublinear expectation which clearly satisfies (A) and (B).
	Assume that there exists a sequence of compact sets $K_n\subset\Omega$ such that $\mathcal{E}(1_{K_n^c}|\mathcal{G})\downarrow 0$ pointwise, where one may assume without loss of generality that $K_n=C_n\times C_n$ for $C_n\subset \mathbb{R}^\mathbb{N}$ compact.
	Since $\mathcal{E}(1_{K_n^c}|\mathcal{G})(\omega)=1_{C_n^c}(\omega_1)$ it follows that $\mathbb{R}^\mathbb{N}=\bigcup\{ C_n : n\in\mathbb{N}\}$.
	Thus, by the Baire category theorem, there exists some $n$ such that $C_n$ has nonempty interior.
	However, every compact subset of $\mathbb{R}^\mathbb{N}$ clearly has empty interior, and thus such a sequence $K_n$ cannot exist.
\end{proof}	

\begin{theorem}
	\label{thm:cond.rep}
	Let $\mathcal{E}(\cdot|\mathcal{G})$ be a conditional nonlinear expectation which satisfies (A) and (B).
	Then there exists a $\mathcal{G}\otimes\mathcal{B}(\mathfrak{P}(\Omega))$-lower semianalytic function $\alpha_\mathcal{G}\colon \Omega \times\mathfrak{P}(\Omega)\to[0,+\infty]$ such that for every $\omega\in\Omega$ it holds that $\alpha_\mathcal{G}(\omega,\cdot)$ is convex, $\inf_{P} \alpha_\mathcal{G}(\omega,\cdot)=0$, for every $c\in\mathbb{R}$ the set $\{\alpha_\mathcal{G}(\omega,\cdot)\leq c\}$ is compact, $\alpha_\mathcal{G}(\omega,P)<+\infty$ implies $P([\omega]_\mathcal{G})=1$, and one has
	\begin{align}
	\label{eq:cond.rep}
	\mathcal{E}(X|\mathcal{G})(\omega)=\sup_{P\in\mathfrak{P}(\Omega)} ( E_P[X] - \alpha_\mathcal{G}(\omega,P) ) 
	\end{align}
	for every $\omega\in\Omega$ and $X\in\mathcal{L}(\mathcal{F})$.
		
	Conversely, if $\alpha_\mathcal{G}\colon \Omega\times\mathfrak{P}(\Omega)\to[0,+\infty]$ is a $\mathcal{G}\otimes\mathcal{B}(\mathfrak{P}(\Omega))$-lower semianalytic function such that
	$\inf_P \alpha_\mathcal{G}(\omega,P)= 0$ and $\alpha_\mathcal{G}(\omega,P)<+\infty$ implies $P([\omega]_\mathcal{G})=1$ for every $\omega\in\Omega$, then $\mathcal{E}(\cdot|\mathcal{G})$ defined by \eqref{eq:cond.rep} is a conditional nonlinear expectation satisfying (B). 
	If in addition $\{\alpha_\mathcal{G}(\omega,\cdot)\leq c\}$ is compact and $\alpha_\mathcal{G}(\omega,\cdot)$ is convex for every $\omega\in\Omega$ and $c\in\mathbb{R}$, then $\mathcal{E}(\cdot|\mathcal{G})$ also satisfies (A).
\end{theorem}

\begin{remark}
\label{rem:unique.penalty}
	For a conditional nonlinear expectation $\mathcal{E}(\cdot|\mathcal{G})$ satisfying (A) and (B) there are in general many functions $\alpha_\mathcal{G}\colon \Omega\times\mathfrak{P}(\Omega)\to[0,+\infty]$ such that \eqref{eq:cond.rep} holds, and not every $\alpha_\mathcal{G}$ needs to be lower semianalytic.
	However, if $\alpha_\mathcal{G}(\omega,\cdot)$ is required to be convex and lower semicontinuous (which particularly is satisfied if all sublevel sets are compact) for every $\omega\in \Omega$, then $\alpha_\mathcal{G}$ is unique, $\mathcal{G}\otimes\mathcal{B}(\mathfrak{P}(\Omega))$-lower semianalytic, and in fact given by
	\begin{align}
	\label{eq:def.alpha}
	\alpha_\mathcal{G}(\omega,P)=\sup_{X\in C_b(\Omega)} ( E_P[X]-\mathcal{E}(X|\mathcal{G})(\omega) )
	\end{align}
	for $\omega\in\Omega$ and $P\in\mathfrak{P}(\Omega)$.
\end{remark}
\begin{proof}
	To show that $\alpha_\mathcal{G}$ needs not to be unique (or $\mathcal{G}\otimes\mathcal{B}(\mathfrak{P}(\Omega))$-lower semianalytic), let $\Omega=[0,1]^2$ with $\mathcal{G}$ generated by the first coordinate.
	Define $\alpha_\mathcal{G}(\omega,P):=0$ for $\omega\in \Omega$ and $P\in\mathfrak{P}(\Omega)$ with $P([\omega]_\mathcal{G})=1$ and $+\infty$ else.
	By Theorem \ref{thm:cond.rep}
	\[\mathcal{E}(X|\mathcal{G})(\omega):=\sup_{P\in\mathfrak{P}(\Omega)}(E_P[X]-\alpha_\mathcal{G}(\omega,P) )
	=\sup_{x\in[0,1]} X(\omega_1,x) \]
	for $\omega=(\omega_1,\omega_2)\in\Omega$ defines a conditional nonlinear expectation $\mathcal{E}(\cdot|\mathcal{G})\colon\mathcal{L}(\mathcal{F})\to\mathcal{L}(\mathcal{G})$ which satisfies (A) and (B).
	Now let $R\colon [0,1]\to\mathfrak{\Omega}$ be a non-universally measurable function such that $R(x)(\{x\}\times[0,1])=1$ but $R(x)\notin D(x)$ for all $x\in[0,1]$, where $D(x)$ denotes the set of all probabilities $P$ such that $P(\{x\}\times \{y\})=1$  for some $y\in[0,1]$.
	Then, for $\tilde{\alpha}_\mathcal{G}(\omega,P)=+\infty 1_{D(\omega_1)^c\cap \{R(\omega_1)\}^c}(P)$, one has $\mathcal{E}(X|\mathcal{G})(\omega)=\sup_{P\in\mathfrak{P}(\Omega)}(E_P[X]-\tilde{\alpha}_\mathcal{G}(\omega,P))$.
	However, $\tilde{\alpha}_\mathcal{G}$ is not $\mathcal{G}$-lower semianalytic as $\{\alpha_\mathcal{G}\leq 0\}$ is the disjoint union of $\{(\omega,P) : P\in D(\omega_1)\}$ and $\{(\omega,R(\omega_1)):\omega\in \Omega\}$.

	As for the second part, notice that \eqref{eq:def.alpha} follows from the Fenchel-Moreau theorem. 
	It is shown within the proof of Theorem \ref{thm:cond.rep} that under conditions (A) and (B), the function $\alpha_\mathcal{G}$ defined by \eqref{eq:def.alpha} is $\mathcal{G}\otimes\mathcal{B}(\mathfrak{P}(\Omega))$-lower semianalytic.
\end{proof}

\begin{proof}[\text{\bf Proof of Theorem \ref{thm:cond.rep}}]
	Let $\mathcal{E}(\cdot|\mathcal{G})$ be a conditional nonlinear expectation which satisfies (A) and (B), and fix some $\omega\in \Omega$.
	Then the functional $\mathcal{E}(\cdot|\mathcal{G})(\omega)$ from $\mathcal{L}(\mathcal{F})$ to $\mathbb{R}$ is a nonlinear expectation which satisfies all assumptions of Choquet's theorem (more precisely, the convex version as stated in \cite[Section 2]{bartl2017robust}). 
	Thus
	\begin{align}
	\label{eq:rep.choquet.in.proof}
	\mathcal{E}(X|\mathcal{G})(\omega)=\sup_{P\in \mathfrak{P}(\Omega)} 
	(E_{P}[X]-\alpha_\mathcal{G}(\omega,P))
	\quad\text{for all }X\in \mathcal{L}(\mathcal{F}),
	\end{align}
	where $\alpha_\mathcal{G}(\omega,P):=\sup_{X\in C_b(\Omega)} (E_{P}[X]-\mathcal{E}(X|\mathcal{G})(\omega))$.
	By \cite[Theorem 2.2]{bartl2017robust} the set $\{\alpha_\mathcal{G}(\omega,\cdot)\leq c\}$ is compact for every $c\in\mathbb{R}$ and, as a supremum over affine functionals, $\alpha_\mathcal{G}(\omega,\cdot)$ is convex.
	To show that $\alpha_\mathcal{G}(\omega,P)<+\infty$ implies that $P([\omega]_\mathcal{G})=1$, let $\lambda\in\mathbb{R}$ be arbitrary and define $Y:=\lambda 1_{[\omega]_\mathcal{G}}\in\mathcal{L}(\mathcal{G})$.
	As \eqref{eq:rep.choquet.in.proof} holds for every $X\in\mathcal{L}(\mathcal{F})$ and $\mathcal{E}(Y|\mathcal{G})(\omega)=Y(\omega)=\lambda$ by assumption, one obtains 
	$\alpha_\mathcal{G}(\omega,P)\geq E_{P}[Y]-\mathcal{E}(Y|\mathcal{G})(\omega)=\lambda (P([\omega]_\mathcal{G})-1)$; thus $P([\omega]_\mathcal{G})=1$.		
	The only thing which now remains open is to show that $\alpha_\mathcal{G}$ is $\mathcal{G}\otimes\mathcal{B}(\mathfrak{P}(\Omega))$-lower semianalytic.
	As $\Omega$ is a Polish space, there exists a metric $d'$ on $\Omega$ inducing the original topology such that the space $uc_b(\Omega,d')$ becomes separable \cite[Lemma 3.1.4]{stroock2010probability}.
	Here $uc_b(\Omega,d')$ denotes the set of all bounded functions from $\Omega$ to $\mathbb{R}$ which are uniformly continuous with respect to $d'$, and this space is endowed with the maximum norm.
	Let $D$ be a countable dense subset.
	As $-\|X\|_\infty\leq\mathcal{E}(X|\mathcal{G})\leq\|X\|_\infty$ for every $X\in\mathcal{L}(\mathcal{F})$	by monotonicity of $\mathcal{E}(\cdot|\mathcal{G})$, it follows for all	$c\in\mathbb{R}$, $\omega\in\Omega$, and $X\in\mathcal{L}(\mathcal{F})$	that
	\begin{align}
	\label{eq:Et.restriction.Lambdac} 
	\mathcal{E}(X|\mathcal{G})(\omega)= \sup_{P\in\Lambda_{2c}(\omega)}	(E_P[X]-\alpha_\mathcal{G}(\omega,P) )	
	\quad\text{if } \|X\|_\infty< c,
	\end{align}
	where $\Lambda_{2c}(\omega):=\{\alpha_\mathcal{G}(\omega,\cdot)\leq 2c\}$.
	Fix $\omega\in \Omega$, $P\in\mathfrak{P}(\Omega)$, $X\in C_b(\Omega)$, $\varepsilon\in (0,1)$, and let $c:=\|X\|_\infty+1$.
	As the set $\Lambda_{2c}(\omega)\cup\{P\}$ is compact, Prokhorov's theorem yields the existence of a compact set $K\subset \Omega$ such that 
	\[ P(K^c)\leq \frac{\varepsilon}{c}\quad\text{and}\quad 
	\sup_{Q\in\Delta_{2c}(\omega)} Q(K^c)\leq \frac{\varepsilon}{c}.\]
	Moreover, $X1_K\in uc_b(K,d')$, and by a version of Tietze's extension theorem \cite[Theorem 3]{katvetov1951real} there
	exists a uniformly continuous function $Y\in uc_b(\Omega,d')$ such that $Y=X$ on $K$, where one can assume without loss of generality that $\|Y\|_\infty\leq \|X\|_\infty$.
	Moreover, as $D\subset uc_b(\Omega,d')$ is dense, there exists $Y'\in D$ such that $\|Y-Y'\|_\infty\leq \varepsilon$.
	Then in particular $\|Y'\|_\infty< c$, so that for every $Q\in\Lambda_{2c}(\omega)$, it holds
	\[E_Q[Y']
	\leq	E_Q[Y' 1_K] +\varepsilon
	\leq E_Q[X 1_K]+2\varepsilon
	= E_Q[X] + 3\varepsilon.\]
	In combination with \eqref{eq:Et.restriction.Lambdac}, this implies
	$\mathcal{E}(Y'|\mathcal{G})(\omega)
	\leq \mathcal{E}(X|\mathcal{G})(\omega)+3\varepsilon$.
	Further, changing the roles of $X$ and $Y'$ and replacing $Q$ by $P$, one gets that $E_P[X]\leq E_P[Y']+3\varepsilon$ and therefore
	\begin{align*}
	E_P[X]-\mathcal{E}(X|\mathcal{G})(\omega)
	&\leq E_P[Y']-\mathcal{E}(Y'|\mathcal{G})(\omega) + 6\varepsilon\\
	&\leq \sup_{Z\in D} ( E_P[Z]-\mathcal{E}(Z|\mathcal{G})(\omega) ) + 6\varepsilon.
	\end{align*}
	As $D$ is a subset of $C_b(\Omega)$ and $X\in C_b(\Omega)$, $\varepsilon\in (0,1)$ were arbitrary, this yields
	\[	\alpha_\mathcal{G}(\omega,P)
	=\sup_{Z\in D} ( E_P[Z]-\mathcal{E}(Z|\mathcal{G})(\omega) ).\] 
	Finally, for every $Z\in D$, the function $(\omega,P)\mapsto E_P[Z] -\mathcal{E}(Z|\mathcal{G})(\omega)$
	is $\mathcal{G}\otimes\mathcal{B}(\mathfrak{P}(\Omega))$-lower semianalytic, as the sum of such functions.
	As the countable supremum, $\alpha_\mathcal{G}$ inherits this property.
	
	To prove the second statement, let  $\alpha_\mathcal{G}\colon \Omega\times\mathfrak{P}(\Omega)\to[0,+\infty]$ be a given $\mathcal{G}\otimes\mathcal{B}(\mathfrak{P}(\Omega))$-lower semianalytic function for which $\alpha_\mathcal{G}(\omega,P)<+\infty$ implies $P([\omega]_\mathcal{G})=1$ and $\inf_P \alpha_\mathcal{G}(\omega,P)= 0$ for every $\omega\in\Omega$.
	Define 
	\[\mathcal{E}(X|\mathcal{G})(\omega):=\sup_{P\in\mathfrak{P}(\Omega)} ( E_P[X]-\alpha_\mathcal{G}(\omega,P) )\]
	for $\omega\in\Omega$ and $X\in\mathcal{L}(\mathcal{F})$.
	For fixed $X\in\mathcal{L}(\mathcal{F})$, the mapping
	\[ \Omega\times\mathfrak{P}(\Omega)\to[-\infty,+\infty),\quad
	(\omega,P)\mapsto E_P[X]-\alpha_\mathcal{G}(\omega,P)  \]
	is $\mathcal{G}\otimes\mathcal{B}(\mathfrak{P}(\Omega))$-upper semianalytic \cite[Proposition 7.48]{bertsekas1978stochastic}.
	Therefore, it follows from \cite[Corollary 6.10.10]{bogachev2007measure} that $\omega\mapsto \mathcal{E}(X|\mathcal{G})(\omega)$ is $\mathcal{G}$-upper semianalytic.
	Further, as $\inf_P \alpha_\mathcal{G}(\omega,P)=0$, one has 
	\[	-\|X\|_\infty \leq \mathcal{E}(X|\mathcal{G})(\omega)\leq \|X\|_\infty\]
	so that $\mathcal{E}(X|\mathcal{G})\in\mathcal{L}(\mathcal{G})$.	
	For $X\in\mathcal{L}(\mathcal{G})$ and every $P$ with $\alpha_\mathcal{G}(\omega,P)<+\infty$ one has $E_P[X]=X(\omega)$, therefore $\mathcal{E}(X|\mathcal{G})=X$.
	The other properties needed for $\mathcal{E}(\cdot|\mathcal{G})$ to be a conditional nonlinear expectation are immediate. 
	Condition (B) follows by interchanging two suprema and the monotone convergence theorem (applied to each $E_P[\cdot]$).
	
	Assume in addition that $\alpha_\mathcal{G}(\omega,\cdot)$ is convex and $\Lambda_c(\omega):=\{\alpha_\mathcal{G}(\omega,\cdot)\leq c\}$ is compact for every $c\in\mathbb{R}$ and $\omega\in\Omega$.
	Fix some $\omega\in\Omega$ and a sequence $X_n\in usc_b(\Omega)$ which decreases pointwise to some $X\in \mathcal{L}(\mathcal{F})$.
	Then it follows as in \eqref{eq:Et.restriction.Lambdac} that
	\begin{align}
	\label{eq:Et.restriction.Lambdac.2}
		\mathcal{E}(Y|\mathcal{G})(\omega)
		=\max_{P\in\Lambda_{2c}(\omega)} (E_P[Y]-\alpha_\mathcal{G}(\omega,P))
		\quad \text{for } Y\in\{X,X_1,X_2,\dots\},
	\end{align}
	where
	$c:=\max\{\|X_1\|_\infty,\|X\|_\infty\}+1$.
	As $\Lambda_{2m}(\omega)$ is compact and convex, $X_n$ is a decreasing sequence, and 
	\[\mathfrak{P}(\Omega)\ni P\mapsto E_P[X_n]-\alpha_\mathcal{G}(\omega,P)\]
	is convex and upper semicontinuous for every $n$ (approximate $X_n$ from above by continuous functions), it follows from \eqref{eq:Et.restriction.Lambdac.2}, a minimax theorem \cite[Theorem 2]{fan1953minimax}, and the monotone convergence theorem that
	\[ \inf_{n\in\mathbb{N}}  \mathcal{E}(X_n|\mathcal{G})(\omega)
	=\max_{P\in\Lambda_{2m}(\omega)} \inf_{n\in\mathbb{N}} (E_P[X_n]-\alpha_\mathcal{G}(\omega,P))
	=\mathcal{E}(X|\mathcal{G})(\omega).\]
	Thus condition (A), the missing part of Remark \ref{rem:dual.char.A}, and Remark \ref{rem:cont.above.usc} are proven.
\end{proof}

\begin{proof}[\text{\bf Proof of Theorem \ref{thm:cond.rep.sublin}}]
	If $\mathcal{E}(\cdot|\mathcal{G})$ satisfies (A) and (B) and is sublinear, then Theorem \ref{thm:cond.rep} guarantees that $\mathcal{E}(\cdot|\mathcal{G})$ has the representation \eqref{eq:cond.rep} where $\alpha_\mathcal{G}$ is defined by \eqref{eq:def.alpha} (see the proof of the theorem or Remark \ref{rem:unique.penalty}). 
	Now a scaling argument shows that $\alpha_\mathcal{G}$ only takes the values $0$ or $+\infty$.
	Indeed, one has $\alpha_\mathcal{G}(\omega,P)=\sup_{X\in C_b(\Omega)}(E_P[X]-\mathcal{E}(X|\mathcal{G})(\omega))\geq0$ by plugging in the function $X=0$.
	On the other hand, if $\alpha_\mathcal{G}(\omega,P)>0$, there exists $X\in C_b(\Omega)$ such that $E_P[X]-\mathcal{E}(X|\mathcal{G})(\omega)>0$, hence $\alpha_\mathcal{G}(\omega,P)\geq\sup_{\lambda\geq0} \lambda (E_P[X]-\mathcal{E}(X|\mathcal{G})(\omega))=+\infty$.
	Therefore $\mathcal{E}(X|\mathcal{G})(\omega)=\sup_{P\in\mathcal{P}_\mathcal{G}(\omega)} E_P[X]$ for $\mathcal{P}_\mathcal{G}(\omega):=\{ \alpha_\mathcal{G}(\omega,\cdot)=0\}$.
	Notice that $\mathop{\mathrm{Graph}}(\mathcal{P}_\mathcal{G})=\{\alpha_\mathcal{G}\leq 0\}$ is an $\mathcal{G}\otimes\mathcal{B}(\mathfrak{P}(\Omega))$-analytic set.
	Conversely, if $\mathcal{P}_\mathcal{G}$ is given, just set $\alpha_\mathcal{G}(\omega,P):=+\infty1_{\mathcal{P}_\mathcal{G}(\omega)^c}(P)$.
	Then, since $\{\alpha_\mathcal{G}\leq c\}$ is empty for $c<0$ and equals $\mathop{\mathrm{Graph}}(\mathcal{P}_\mathcal{G})$ otherwise, $\alpha_\mathcal{G}$ is $\mathcal{G}\otimes\mathcal{B}(\mathfrak{P}(\Omega))$-lower semianalytic.
\end{proof}

The following (somewhat auxiliary) result shows that given a nonlinear expectation $\mathcal{E}(\cdot)\colon\mathcal{L}(\mathcal{F})\to\mathbb{R}$ with representing set $\mathcal{P}$, it is possible to give meaning to $\sup_{P\in\mathcal{P}} E_P[\cdot|\mathcal{G}]$.
However, without some assumption on the tower property (as will be investigated in the next section), the resulting conditional nonlinear expectation will have little to do with $\mathcal{E}(\cdot)$.

\begin{remark}
\label{rem:disintegration}
	Assume that $\mathcal{G}=\sigma(\phi)$ for some Borel $\phi\colon\Omega\to S$ where $S$ is another Polish space.
	Then, by \cite[Proposition 7.27]{bertsekas1978stochastic} (which is formulated for product spaces but readily extends to the present setting), one can construct a regular version of the conditional probability such that 
	\[\mathop{\mathrm{Dis}}\colon\Omega\times\mathfrak{P}(\Omega)\to\Omega\times\mathfrak{P}(\Omega),\quad (\omega,P)\mapsto (\omega,P_\mathcal{G}(\omega))\]
	is $\mathcal{G}\otimes\mathcal{B}(\mathfrak{P}(\Omega))$-measurable.
	In particular, if $\mathcal{P}\subset\mathfrak{P}(\Omega)$ is an $\mathcal{B}(\mathfrak{P}(\Omega))$-analytic set and $\mathcal{P}_\mathcal{G}(\omega):=\{P_\mathcal{G}(\omega) : P\in\mathcal{P}\}$ for every $\omega\in\Omega$, then $\mathop{\mathrm{Graph}}(\mathcal{P}_\mathcal{G})=\mathop{\mathrm{Dis}}(\Omega\times \mathcal{P})$ is an $\mathcal{G}\otimes\mathcal{B}(\mathfrak{P}(\Omega))$-analytic set \cite[Corollary 6.10.10]{bogachev2007measure}.
	Therefore $\mathcal{E}(\cdot|\mathcal{G})$ defined through $\mathcal{E}(X|\mathcal{G})(\omega):=\sup_{P\in\mathcal{P}_\mathcal{G}(\omega)} E_P[X]$ maps $\mathcal{L}(\mathcal{F})$ to $\mathcal{L}(\mathcal{G})$.
	In fact, one can also show that $\mathcal{P}_\mathcal{G}'\colon\Omega\rightsquigarrow\mathfrak{P}(\Omega)$ defined as the closed convex hull of $\mathcal{P}_\mathcal{G}(\omega)$ for each $\omega\in\Omega$ has $\mathcal{G}\otimes\mathcal{B}(\mathfrak{P}(\Omega))$-analytic graph.
\end{remark}

\subsection{The tower property}

In contrast to the linear case, the tower property does not need to hold in general.
The goal of this section to characterize on the level of representing sets of probabilities, in which cases the tower property holds true.

\begin{lemma}
\label{lem:composition} 
	Let $\mathcal{E}(\cdot|\mathcal{G})\colon\mathcal{L}(\mathcal{F})\to\mathcal{L}(\mathcal{G})$ be a conditional nonlinear expectation and $\mathcal{E}'(\cdot)\colon\mathcal{L}(\mathcal{G})\to\mathbb{R}$ be a nonlinear expectation such that
	\begin{align*}
	\mathcal{E}'(X)&=\sup_{P\in\mathfrak{P}(\Omega)}( E_P[X]-\beta(P)) & &\text{for } X\in\mathcal{L}(\mathcal{G}), \\
	\mathcal{E}(X|\mathcal{G})(\omega)&=\sup_{P\in\mathfrak{P}(\Omega)}( E_P[X]-\gamma_\mathcal{G}(\omega,P))& &\text{for } X\in\mathcal{L}(\mathcal{F}) 
	\end{align*}
	for some $\beta\colon\mathfrak{P}(\Omega)\to[0,+\infty]$ such that $\beta(P)=\beta(Q)$ if $P=Q$ on $\mathcal{G}$ and $\mathcal{G}\otimes\mathcal{B}(\mathfrak{P}(\Omega))$-lower semianalytic $\gamma_\mathcal{G}\colon\Omega\times\mathfrak{P}(\Omega)\to[0,+\infty]$ such that $\gamma_\mathcal{G}(\omega,P)<+\infty$ implies $P([\omega]_\mathcal{G})=1$ for every $\omega\in\Omega$.
	Then $\mathcal{E}(\cdot):=\mathcal{E}'(\mathcal{E}(\cdot|\mathcal{G}))$ defines a nonlinear expectation from $\mathcal{L}(\mathcal{F})$ to $\mathbb{R}$	and 
	\begin{align}
	\label{eq:rep.composition}
	\mathcal{E}(X)
	=\sup_{P=P\otimes P_\mathcal{G}\in\mathfrak{P}(\Omega)} 
	\Big(E_P[X]-\big(\beta(P)+E_P[\gamma_\mathcal{G}(\cdot,P_\mathcal{G}(\cdot))]\big)\Big)
	\end{align}
	for $X\in\mathcal{L}(\mathcal{F})$.
\end{lemma}
\begin{proof}
	It is clear that 	
	$\mathcal{E}(\cdot):=\mathcal{E}'(\mathcal{E}(\cdot|\mathcal{G}))\colon\mathcal{L}(\mathcal{F})\to\mathbb{R}$
	defines a nonlinear expectation.
	For $P=P\otimes P_\mathcal{G}\in\mathfrak{P}(\Omega)$ and $X\in\mathcal{L}(\mathcal{F})$ one has
	\begin{align*}
	\mathcal{E}(X)
	=\mathcal{E}'(\mathcal{E}(X|\mathcal{G}))
	&\geq E_P[\mathcal{E}(X|\mathcal{G})]-\beta(P)\\
	&\geq E_P[E_{P_\mathcal{G}(\cdot)}[X]-\gamma_\mathcal{G}(\cdot,P_\mathcal{G}(\cdot))]-\beta(P)\\
	&=E_P[X] -(\beta(P)+E_P[\gamma_\mathcal{G}(\cdot,P_\mathcal{G}(\cdot))]),
	\end{align*}
	which shows that the left hand side in \eqref{eq:rep.composition} is larger than the right hand side.

	To prove the reverse inequality fix some $X\in\mathcal{L}(\mathcal{F})$ and $\varepsilon>0$, and let $Q\in\mathfrak{P}(\Omega)$ be such that
	\begin{align}
	\label{eq:tower.weak.epsilon}
	\mathcal{E}'(\mathcal{E}(X|\mathcal{G}))
	\leq E_Q[\mathcal{E}(X|\mathcal{G})] -\beta(Q)+\varepsilon. 
	\end{align}
	By \cite[Proposition 7.48]{bertsekas1978stochastic} the mapping
	\[ \Omega\times\mathfrak{P}(\Omega)\to[-\infty,+\infty),
	\quad (\omega,P)\mapsto E_P[X] -\gamma_\mathcal{G}(\omega,P)\]
	is $\mathcal{G}\otimes\mathcal{B}(\mathfrak{P}(\Omega))$-upper semianalytic and as $X$ is bounded from above, by \cite[Theorem 6.9.12]{bogachev2007measure}, there exists a $\mathcal{G}$-universally measurable mapping $R\colon \Omega\to \mathfrak{P}(\Omega)$ such that
	\[ \mathcal{E}(X|\mathcal{G})(\omega)
	=\sup_{P\in\mathfrak{P}(\Omega)} (E_P[X] -\gamma_\mathcal{G}(\omega,P))
	\leq E_{R(\omega)}[X] -\gamma_\mathcal{G}(\omega,R(\omega))+ \varepsilon \]
	for every $\omega\in \Omega$.
	Define $P:=Q\otimes R\in\mathfrak{P}(\Omega)$.
	As $R(\omega)([\omega]_\mathcal{G})=1$ for all $\omega\in\Omega$, one has that $P=Q$ on $\mathcal{G}$, and as $R$ is measurable w.r.t.~the $P$-completion of $\mathcal{G}$, it is a version of the conditional disintegration, that is $R=P_\mathcal{G}$ $P$-almost surely.
	Together with \eqref{eq:tower.weak.epsilon} this implies
	\begin{align*} 
	\mathcal{E}'(\mathcal{E}(X|\mathcal{G}))
	&\leq E_{Q}[E_{R(\cdot)}[X] -\gamma_\mathcal{G}(\cdot,R(\cdot))+ \varepsilon] -\beta(Q)+\varepsilon\\
	&=E_P[X] - \big(\beta(P)+E_P[\gamma_\mathcal{G}(\cdot,P_\mathcal{G}(\cdot))]\big)+2\varepsilon.
	\end{align*}
	As $\varepsilon$ was arbitrary, the claim follows.
\end{proof}

\begin{remark}
\label{rem:gen.tight.not.E.tight}
	If in the above lemma both $\mathcal{E}'(\cdot)$ and $\mathcal{E}(\cdot|\mathcal{G})$ satisfy condition (A), then $\mathcal{E}'(\mathcal{E}(\cdot|\mathcal{G}))$ does not necessary satisfy (A).
\end{remark}
\begin{proof}
	Let $\Omega=\mathbb{R}^2$ with $\mathcal{G}$ generated by the first coordinate, and define 
	\[ \mathcal{E}'(X):=\sup_{x\in[0,1]} X(x) \quad\text{and}\quad
	\mathcal{E}(X|\mathcal{G})(\omega):=X\Big(\omega_1,\frac{1}{\omega_1}\Big) 
	\quad\text{with } \frac{1}{0}:=0\]
	for $X\in\mathcal{L}(\mathcal{G})$ and $X\in\mathcal{L}(\mathcal{F})$, respectively, and $\omega=(\omega_1,\omega_2)\in\Omega$.
	Since $x\mapsto(x,1/x)$ is Borel, it follows from \cite[Lemma 7.30]{bertsekas1978stochastic} that $\mathcal{E}(X|\mathcal{G})\in\mathcal{L}(\mathcal{G})$.
	The sequence of functions $X_n\in C_b(\Omega)$ defined by 
	\[X_n(\omega):=(\omega_2-n+1)1_{[n-1,n]}(\omega_2)+1_{(n,\infty)}(\omega_2)\quad\text{for } \omega=(\omega_1,\omega_2)\in \Omega\]
	satisfies $X_n\downarrow0$ but 
	\[\mathcal{E}'(\mathcal{E}(X_n|\mathcal{G}))
	=\sup_{x\in[0,1]} X_n\Big(x,\frac{1}{x}\Big)
	\geq X_n\Big(\frac{1}{n},n\Big)
	=1 \]
	for all $n$.
	Hence $\mathcal{E}'(\mathcal{E}(\cdot|\mathcal{G}))$ does not satisfy condition (A), while both $\mathcal{E}'(\cdot)$ and $\mathcal{E}(\cdot|\mathcal{G})$ clearly do satisfy (A).	
\end{proof}

In Lemma \ref{lem:composition} is was shown that the composition of nonlinear expectations can be 
represented by a function which equals the (integrated) sum over $\beta$ and $\gamma_\mathcal{G}$.
However, it was not shown that this function is minimal in the sense that it is lower semicontinuous in $P$
and thus given by formula \eqref{eq:def.alpha}.
The following theorem shows that, given additional regularity, this is true.	

\begin{theorem}
	\label{thm:char.tower}
	Let $\mathcal{E}(\cdot)\colon\mathcal{L}(\mathcal{F})\to\mathbb{R}$ and 
	$\mathcal{E}(\cdot|\mathcal{G})\colon\mathcal{L}(\mathcal{F})\to\mathcal{L}(\mathcal{G})$ be two (conditional) nonlinear expectations which satisfy (A) and (B) and therefore 
	\begin{align*}
	\mathcal{E}(X)&=\sup_{P\in\mathfrak{P}(\Omega)} (E_P[X]-\alpha(P))
 	& & \text{for }X\in\mathcal{L}(\mathcal{F}),\\
	\mathcal{E}(X|\mathcal{G})(\omega)&=\sup_{P\in\mathfrak{P}(\Omega)} (E_P[X]-\gamma_\mathcal{G}(\omega,P))
	& & \text{for }X\in\mathcal{L}(\mathcal{F}),
	\end{align*}
	where
	$\alpha\colon\mathfrak{P}(\Omega)\to[0,+\infty]$ is lower semicontinuous and
	$\gamma_\mathcal{G}\colon\Omega\times\mathfrak{P}(\Omega)\to[0,+\infty]$
	is as in the first part of Theorem \ref{thm:cond.rep}.
	Define 
	$\beta(P):=\sup_{X\in\mathcal{L}(\mathcal{G})} (E_P[X]-\mathcal{E}(X))$ for $P\in\mathfrak{P}(\Omega)$.
	Then
	\[ \mathcal{E}(\cdot)\leq \mathcal{E}(\mathcal{E}(\cdot|\mathcal{G}))
	\quad\text{if and only if}\quad
	\alpha(P)
	\geq \beta(P) + E_P[\gamma_\mathcal{G}(\cdot,P_\mathcal{G}(\cdot))]\] 
	for all $P=P\otimes P_\mathcal{G}\in\mathfrak{P}(\Omega)$.
	Assume further that $\mathcal{E}(X|\mathcal{G})$ is $\mathcal{G}$-measurable for every $X\in C_b(\Omega)$.
	Then 
	\[ \mathcal{E}(\cdot)\geq \mathcal{E}(\mathcal{E}(\cdot|\mathcal{G}))
	\quad\text{if and only if}\quad
	\alpha(P)
	\leq \beta(P) + E_P[\gamma_\mathcal{G}(\cdot,P_\mathcal{G}(\cdot))]\] 
	for all $P=P\otimes P_\mathcal{G}\in\mathfrak{P}(\Omega)$.
\end{theorem}

\begin{remark}
\label{rem:beta.easy.form}
	If in the setting of Theorem \ref{thm:char.tower} the function $\alpha$ has compact sublevel sets $\{\alpha\leq c\}$ for every $c\in\mathbb{R}$ and $\mathcal{G}=\sigma(\phi)$ for some continuous $\phi\colon \Omega\to S$ where $S$ is another Polish space, then the function $\beta$ has the following more intuitive formula
	\[\beta(P)=\inf\{ \alpha (Q) : Q\in\mathfrak{P}(\Omega) \text{ such that } Q=P \text{ on } \mathcal{G}\} \quad\text{for }P\in\mathfrak{P}(\Omega).\]
	Moreover, $\beta$ is convex, the infimum over $Q$ is attained, and $\{\beta\leq c\}$ is compact in the (not Hausdorff) topology $\sigma(\mathfrak{P}(\Omega), C_b(\Omega)\cap \mathcal{L}(\mathcal{G}))$ for every $c\in\mathbb{R}$.
\end{remark}
\begin{proof}
	It follows from the definition of $\beta(P)=\sup_{X\in\mathcal{L}(\mathcal{G})} (E_P[X]-\mathcal{E}(X))$ that $\beta$ is convex and $\beta(P)=\beta(Q)$ if $P=Q$ on $\mathcal{G}$.
	At the beginning of the proof of Theorem \ref{thm:char.tower} below, it will be shown that $\beta(Q)\leq\alpha(Q)$ for all $Q$ from which it follows that $\beta(P)\leq  \inf\{ \alpha (Q) : Q=P \text{ on } \mathcal{G}\}$.
	On the other hand, let $C:=C_b(\Omega)\cap\mathcal{L}(\mathcal{G})$, so that
	\[\beta(P)\geq\sup_{X\in C} (E_P[X]-\mathcal{E}(X))
	=\sup_{X\in C}\inf_{Q\in\mathfrak{P}(\Omega)} (E_P[X]-E_Q[X]+\alpha(Q)).\]
	Using a minimax theorem \cite[Theorem 2]{fan1953minimax} and the fact that $P=Q$ on $\mathcal{G}$ if and only if $E_P[X]=E_Q[X]$ for all $X\in C$ (this follows e.g.~from a monotone class theorem) one obtains that $\beta(P)\geq \inf\{ \alpha (Q) : Q=P \text{ on } \mathcal{G}\}$; hence the claimed formula for $\beta$ holds.
	To show that the infimum is attained assume $\beta(P)<+\infty$ and let $Q_n$ be a minimizing sequence.
	Due to compactness of $\{\alpha\leq c\}$ for all $c\in\mathbb{R}$, there is $Q$ and subsequence still denoted by $Q_n$ such that $Q_n\to Q$. 
	As $\sigma (C)=\mathcal{G}$, a monotone class argument implies $Q=P$ on $\mathcal{G}$.
	The claim now follows from lower semicontinuity of $\alpha$.
	As for compactness, let $P_\theta$ be a net in $\{\beta\leq c\}$.
	Then, by definition, there is a net $Q_\theta$ in $\{\alpha\leq c\}$ with $Q_\theta=P_\theta$ on $\mathcal{G}$.
	By compactness of $\{\alpha\leq c\}$, there is $Q$ and a subnet still denoted by $Q_\theta$ with $Q_\theta\to Q$.
	Let $P:=Q\in\{\beta\leq c\}$.
	Then, for every $X\in C_b(\Omega)\cap\mathcal{L}(\mathcal{G})$, one has $E_P[X]=E_Q[X]=\lim_\theta E_{Q_\theta}[X]=\lim_\theta E_{P_\theta}[X]$ showing that $P_\theta\to P$ in $\sigma(\mathfrak{P}(\Omega), C_b(\Omega)\cap\mathcal{L}(\mathcal{G}))$.
\end{proof}

\begin{proof}[\text{\bf Proof of Theorem \ref{thm:char.tower}}]
	In a first step observe that
	\begin{align}
	\label{eq:rep.E.on.G} 
	\mathcal{E}(X)=\sup_{P\in\mathfrak{P}(\Omega)} (E_P[X]-\beta(P))
	\quad\text{for } X\in\mathcal{L}(\mathcal{G}).
	\end{align}
	Indeed, the left hand side in \eqref{eq:rep.E.on.G} is smaller than the right hand side by definition of $\beta(P)$.
	On the other hand, \eqref{eq:rep.E.on.G} holds by assumption when $\beta$ is replaced by $\alpha$.
	As
	\[\alpha(P)
	=\sup_{X\in C_b(\Omega)}(E_P[X]-\mathcal{E}(X))
	=\sup_{X\in \mathcal{L}(\mathcal{F})}(E_P[X]-\mathcal{E}(X))
	\geq\beta(P)\]
	where the first equality holds by Remark \ref{rem:unique.penalty} due to lower semicontonuty of $\alpha$ and the second one due to the dual presentation $\mathcal{E}(X)=\sup_{P\in\mathfrak{P}(\Omega)}(E_P[X]-\alpha(P))$, it follows that \eqref{eq:rep.E.on.G} holds true.
	Now define 
	\[\delta(P):=\beta(P)+E_P[\gamma_\mathcal{G}(\cdot,P_\mathcal{G}(\cdot))]
	\quad\text{for } P=P\otimes P_\mathcal{G}\in\mathfrak{P}(\Omega).\]		
	It follows from Lemma \ref{lem:composition} that
	\begin{align}
	\label{eq:comp.rep.beta} 
	\mathcal{E}(\mathcal{E}(X|\mathcal{G}))
	=\sup_{P\in\mathfrak{P}(\Omega)} (E_P[X]-\delta(P))
	\quad\text{for } X\in\mathcal{L}(\mathcal{F}).
	\end{align}	
	In particular if $\alpha\leq\delta$, then $\mathcal{E}(\cdot)\geq \mathcal{E}(\mathcal{E}(\cdot|\mathcal{G}))$; and if $\alpha\geq\delta$, then $\mathcal{E}(\cdot)\leq \mathcal{E}(\mathcal{E}(\cdot|\mathcal{G}))$.		

	Now assume that $\mathcal{E}(\cdot)\geq \mathcal{E}(\mathcal{E}(\cdot|\mathcal{G}))$ and fix some $P\in\mathfrak{P}(\Omega)$. 
	By \eqref{eq:comp.rep.beta} it holds
	\[ \mathcal{E}(X)
	\geq \mathcal{E}(\mathcal{E}(X|\mathcal{G}))
	\geq E_P[X]-\delta(P) 
	\quad\text{for } X\in\mathcal{L}(\mathcal{F})\]
	and therefore
	\[ \alpha(P)
	=\sup_{X\in C_b(\Omega)} (E_P[X]-\mathcal{E}(X))
	\leq \sup_{X\in C_b(\Omega)} (E_P[X]-(E_P[X]-\delta(P)))
	=\delta(P).\]
	As $P$ was arbitrary, this shows $\alpha\leq\delta$.	
	
	It remains to prove that $\alpha\geq\delta$ if $\mathcal{E}(\cdot)\leq\mathcal{E}(\mathcal{E}(\cdot|\mathcal{G}))$	and $\mathcal{E}(X|\mathcal{G})$ is $\mathcal{G}$-measurable for every $X\in C_b(\Omega)$.
	To that end, one may argue as in the proof of Theorem \ref{thm:cond.rep} and choose a metric $d'$ on $\Omega$ under which the 
	space of bounded and uniformly continuous functions $uc_b(\Omega,d')$ becomes separable.
	Let $D$ be a countable dense subset and define $D':=\{X-q : X\in D \text{ and } q\in\mathbb{Q}\}$.
	For each $n$, let $D_n\subset D'$ such that $D_n$ consist of exactly $n$ elements, $0\in D_n\subset D_{n+1}$, and $\bigcup\{ D_n : n\in\mathbb{N}\}=D'$.
	Then it holds that
	\begin{align}
	\label{eq:alpha.equals.sup.alphan}
	\gamma_\mathcal{G}(\omega,P)
	&=\sup_n \gamma_\mathcal{G}^n(\omega,P) \quad\text{for } \omega\in \Omega\text{ and } P\in\mathfrak{P}(\Omega),
	\text{ where}\\
	\gamma_\mathcal{G}^n(\omega,P)
	&:=\max\{E_{P}[X]:
	X\in D_n \text{ such that } \mathcal{E}(X|\mathcal{G})(\omega)\leq 0\} \nonumber
	\end{align}
	for every $n$.
	Indeed, as $D_n\subset C_b(\Omega)$, it follows that $\gamma_\mathcal{G}(\omega,P)\geq\sup_n\gamma_\mathcal{G}^n(\omega,P)$.
	For the converse inequality, due to lower semicontinuity of $\gamma_\mathcal{G}(\omega,\cdot)$, one has
	\[ 	\gamma_\mathcal{G}(\omega,P)=\sup_{X\in C_b(\Omega)} (E_P[X]-\mathcal{E}(X|\mathcal{G})(\omega)), \]
	see Remark \ref{rem:unique.penalty}.
	Let $X\in C_b(\Omega)$ and $\varepsilon>0$ be arbitrary.
	It then follows as in the proof of Theorem \ref{thm:cond.rep} that there exists some $Y\in D$ such that
	\[E_P[Y]-\mathcal{E}(Y|\mathcal{G})(\omega)
	\geq E_P[X]-\mathcal{E}(X|\mathcal{G})(\omega)-\varepsilon.\]
	Now let $q$ be rational such that $0\leq q-\mathcal{E}(Y|\mathcal{G})(\omega)\leq\varepsilon$ and define $Z:=Y-q$. 
	Then $\mathcal{E}(Z|\mathcal{G})(\omega)\leq 0$ and 
	\[E_P[Z]
	\geq E_P[Y] - \mathcal{E}(Y|\mathcal{G})(\omega)-\varepsilon
	\geq E_P[X] - \mathcal{E}(X|\mathcal{G})(\omega)-2\varepsilon \]
	Since $Z\in D_n$ for some large $n$, it follows that
	\[\sup_n\gamma_\mathcal{G}^n(\omega,P)
	\geq E_P[Z]
	\geq  E_P[X]-\mathcal{E}(X|\mathcal{G})(\omega)-2\varepsilon\] 
	and, as $X\in C_b(\Omega)$ was arbitrary, $\sup_n\gamma_\mathcal{G}^n(\omega,P)\geq \gamma_\mathcal{G}(\omega,P)-2\varepsilon$.
	This establishes \eqref{eq:alpha.equals.sup.alphan}.

	For the remainder fix some $P\in\mathfrak{P}(\Omega)$.
	Note that $0\in D_n$ and $D_n\subset D_{n+1}$ imply $0\leq \gamma_\mathcal{G}^n\leq\gamma_\mathcal{G}^{n+1}$ for all $n$,  hence the monotone convergence theorem yields
	\begin{align}
	\label{eq:beta.equal.sumYn}
	\begin{aligned}
	\delta(P)
	&=\beta(P) + E_P[\gamma_\mathcal{G}(\cdot,P_\mathcal{G}(\cdot))]\\
	&=\sup_{Y\in \mathcal{L}(\mathcal{G})} (E_Q[Y]-\mathcal{E}(Y)) +\sup_n E_P[\gamma_\mathcal{G}^n(\cdot,P_\mathcal{G}(\cdot))].
	\end{aligned}
	\end{align} 	
	Fix some $Y\in \mathcal{L}(\mathcal{G})$ and $n$, and let $D_n=\{X_1,\dots,X_n\}$ be an enumeration of the set $D_n$.
	Define $i\colon\Omega\to\{1,\dots,n\}$ by
	\[ i(\omega):=\min\big\{ i\in\{1,\dots,n\} : E_{P_\mathcal{G}(\omega)}[X_i]=\gamma_\mathcal{G}^n(\omega,P_\mathcal{G}(\omega))
	\text{ and } \mathcal{E}(X_i|\mathcal{G})(\omega)\leq 0 \big\}. \]
	The three terms inside the minimum which depend on $\omega$ are $\mathcal{G}$-measurable, hence $i$ is also $\mathcal{G}$-measurable.
	Define the function 
	\[Z(\omega):=Y(\omega)+X_{i(\omega)}(\omega)  
	\quad\text{for } \omega\in \Omega.\]
	Then $Z$ is $\mathcal{F}$-upper semianalytic and bounded (as $D_n$ is a finite set), that is, $Z\in\mathcal{L}(\mathcal{F})$.
	For every $P$ with $P([\omega]_\mathcal{G})=1$ one has $E_P[Z]=Y(\omega)+E_P[X_{i(\omega)}(\cdot)]$, hence the dual representation of $\mathcal{E}(\cdot|\mathcal{G})$ and choice of $X_{i(\omega)}$ imply that 
	\[\mathcal{E}(Z|\mathcal{G})(\omega)
	=Y(\omega)+\mathcal{E}(X_{i(\omega)}|\mathcal{G})(\omega)
	\leq Y(\omega).\]
	Therefore, by monotonicity of $\mathcal{E}(\cdot)$ and the assumption that  $\mathcal{E}(\cdot)\leq\mathcal{E}(\mathcal{E}(\cdot|\mathcal{G}))$, one obtains
	$\mathcal{E}(Z)
	\leq \mathcal{E}(\mathcal{E}(Z|\mathcal{G}))
	\leq \mathcal{E}(Y)$
	so that
	\begin{align*}
	E_P[Z]-\mathcal{E}(Z)
	&\geq E_P[Y]+E_{P(d\omega)}[E_{P_\mathcal{G}(\omega)}[X_{i(\omega)}(\cdot)]-\mathcal{E}(Y)\\
	&=E_P[Y]-\mathcal{E}(Y)+E_{P}[\gamma_\mathcal{G}^n(\cdot,P_\mathcal{G}(\cdot))].
	\end{align*}
	As $\mathcal{E}(Z)\geq E_P[Z]-\alpha(P)$ by the dual representation of $\mathcal{E}(\cdot)$, one has $\alpha(P)\geq E_P[Z]-\mathcal{E}(Z)$ and, as $Y$ and $n$  in  \eqref{eq:beta.equal.sumYn} were arbitrary, it follows that $\alpha(P) \geq \delta(P)$.
	This concludes the proof.
\end{proof}

\begin{proof}[\text{\bf Proof of Theorem \ref{thm:char.tower.sublin}}]
	With the notation of Theorem \ref{thm:char.tower}:
	By Remark \ref{rem:unique.penalty} one has $\mathcal{P}=\{\alpha(\cdot)=0\}$ and $\mathcal{P}_\mathcal{G}(\omega)=\{\gamma_\mathcal{G}(\omega,\cdot)=0\}$.
	As $\mathcal{E}(\cdot)$ is sublinear, $\beta$ also only takes the values $0$ and $+\infty$ (compare with the proof of Theorem \ref{thm:cond.rep.sublin}), and one can set  
	\[\mathcal{Q}:=\{\beta\leq 0\}=\{ Q\in\mathfrak{P}(\Omega) : \text{ there is } P\in\mathcal{P} \text{ with } P=Q\text{ on } \mathcal{G}\},\]
	where the last equality is due to Remark \ref{rem:beta.easy.form}.
	Now, by Theorem \ref{thm:char.tower}, one has $\mathcal{E}(\cdot)=\mathcal{E}(\mathcal{E}(\cdot|\mathcal{G}))$ if and only if $\alpha(P)=\beta(P)+E_P[\gamma_\mathcal{G}(\cdot,P_\mathcal{G}(\cdot))]$ for all $P=P\otimes P_\mathcal{G}\in\mathfrak{P}(\Omega)$.
	By definition $\beta(P)+E_P[\gamma_\mathcal{G}(\cdot,P_\mathcal{G}(\cdot))]=0$ if and only if $P\in\mathcal{Q}$ and, $P$-almost surely, $P_\mathcal{G}(\cdot)\in\mathcal{P}_\mathcal{G}(\cdot)$, so that $\mathcal{E}(\cdot)=\mathcal{E}(\mathcal{E}(\cdot|\mathcal{G}))$ if and only if 
	\[\mathcal{P}=\{  P\in\mathfrak{P}(\Omega) : P\in\mathcal{Q} \text{ and } P_\mathcal{G}(\cdot)\in\mathcal{P}_\mathcal{G}(\cdot) \,\,P\text{-almost surely} \}. \]
	Finally, as $\mathcal{Q}\otimes\mathcal{P}_\mathcal{G}=\mathcal{P}\otimes\mathcal{P}_\mathcal{G}$, the claim follows.
\end{proof}

\begin{proposition}
	Let $\mathcal{E}(\cdot|\mathcal{G})$ be a conditional sublinear expectation which satisfies (A) and (B)
	and therefore $\mathcal{E}(X|\mathcal{G})(\omega)=\sup_{P\in\mathcal{P}_\mathcal{G}(\omega)}E_P[X]$
	for some $\mathcal{P}_\mathcal{G}$ as in Theorem \ref{thm:cond.rep.sublin}.
	If $\Omega$ is compact, then $\mathcal{E}(X|\mathcal{G})$ is $\mathcal{G}$-measurable for $X\in C_b(\Omega)$ if and only if $\mathcal{P}_\mathcal{G}$ is weakly $\mathcal{G}$-measurable, that is, for every open set $O\subset\mathfrak{P}(\Omega)$, the weak inverse $\{\omega\in\Omega : \mathcal{P}_\mathcal{G}(\omega)\cap O\neq\emptyset\}$ is in $\mathcal{G}$.  
\end{proposition}
\begin{proof}
	Assume first that $\mathcal{P}_\mathcal{G}\colon \Omega\rightsquigarrow\mathfrak{P}(\Omega)$ is weakly $\mathcal{G}$-measurable.
	Then, since it has nonempty, convex, and compact values, it admits a Castaing representation \cite[Corollary 18.14]{aliprantis2006infinite}: 
	There are $\mathcal{G}$-measurable mappings $R_n\colon \Omega\to\mathfrak{P}(\Omega)$  such that the closure of $\{ R_n(\omega) : n\in\mathbb{N}\}$ equals $\mathcal{P}_\mathcal{G}(\omega)$ for every $\omega\in\Omega$.
	Therefore
	\[ \mathcal{E}(X|\mathcal{G})(\omega)=\sup_n E_{R_n(\omega)}[X]
	\quad\text{for }\omega\in\Omega \text{ and } X\in C_b(\Omega)\]	
	and as $\omega\mapsto E_{R_n(\omega)}[X]$ is Borel \cite[Lemma 7.30]{bertsekas1978stochastic} for every $n$, the claim follows.

	On the other hand, assume that $\mathcal{E}(X|\mathcal{G})$ is $\mathcal{G}$-measurable for every $X\in C_b(\Omega)$ and let $O\subset\mathfrak{P}(\Omega)$ be open.
	Since the weak topology on $\mathfrak{P}(\Omega)$ is locally convex and metrizable, there are closed and convex sets $C_n\subset\mathfrak{P}(\Omega)$ such that $O=\bigcup\{ C_n :n\in\mathbb{N}\}$.
	Therefore $\mathcal{P}_\mathcal{G}(\omega)\cap O = \emptyset$ if and only if $\mathcal{P}_\mathcal{G}(\omega)\cap C_n = \emptyset$ for all $n$.
	Now, as $\mathcal{P}_\mathcal{G}(\omega)$ is compact and convex for every $\omega\in\Omega$ by assumption, the hyper plane separation theorem yields that $\mathcal{P}_\mathcal{G}(\omega)\cap C_n=\emptyset$ if and only if  
	\[\mathcal{E}(X|\mathcal{G})(\omega)
	=\sup_{P\in\mathcal{P}_\mathcal{G}(\omega)}E_P[X]
	<\inf_{P\in C_n} E_P[X]\]
	for some $X\in C_b(\Omega)$.
	As $C_b(\Omega)$ is separable (w.r.t.~to the maximum norm), $X$ can in fact be chosen in some fixed (i.e.~independent of $\omega$ and $n$) countable dense set $D\subset C_b(\Omega)$.
	Therefore
	\[\{\omega\in\Omega : \mathcal{P}_\mathcal{G}(\omega)\cap O = \emptyset\}
	=\bigcap_n\bigcup_{X\in D} \Big\{\omega\in\Omega :\mathcal{E}(X|\mathcal{G})(\omega)< \inf_{P\in C_n} E_P[X]\Big\}. \]
	By assumption all sets on the right hand side are in $\mathcal{G}$, hence the countable union and intersection is in $\mathcal{G}$, too.
\end{proof}

This section ends with an explanation why the upper semianalytic functions (instead of e.g.~Borel functions) are the natural domain and range for conditional nonlinear expectations.
For an illustration on the basis of a concrete example (the $G$-Brownian motion) see \cite{nutz2013constructing}, in particular Section 5.3 and Section 5.4 therein.

\begin{remark}
\label{rem:why.usa}
	As already mentioned in the introduction, Theorem \ref{thm:cond.rep.sublin} should be a extension of Choquet's theorem to the conditional case.
	Therefore it should at least cover the case of conditional nonlinear expectations which are represented by a set-valued mapping $\mathcal{P}_\mathcal{G}$ which depends as little on $\omega$ as possible.
	However, already in this setting it makes little sense to work with linear spaces instead
	of the semianalytic functions.
		
	For example, let $\Omega:=[0,1]^3$ equipped with $\mathcal{G}$ generated by the first two coordinates and $\mathcal{H}$ generated by the first coordinate.
	Define
	\[\mathcal{E}(X|\mathcal{G})(\omega):=\sup_{x\in[0,1]} X(\omega_1,\omega_2,x),
	\hspace{0.5em}
	\mathcal{E}(X|\mathcal{H})(\omega):=\sup_{x,y\in[0,1]} X(\omega_1,x,y) \]
	for $\omega=(\omega_1,\omega_2,\omega_3)\in\Omega$.
	Now assume that there are linear spaces of (bounded) functions $L(\mathcal{F})$, $L(\mathcal{G})$, and $L(\mathcal{H})$	such that $L(\mathcal{F})$ contains all bounded $\mathcal{F}$-measurable functions, and $\mathcal{E}(\cdot|\mathcal{G})$ and $\mathcal{E}(\cdot|\mathcal{H})$ are mappings from $L(\mathcal{F})$ to $L(\mathcal{G})$	and $L(\mathcal{H})$, respectively. 
	By the projective description of analytic sets, every $\mathcal{B}([0,1]^2)$-analytic set $A\subset [0,1]^2$ is the projection of some Borel set $B\subset \Omega$. 
	As $\mathcal{E}(1_B|\mathcal{G})(\omega)=1_A(\omega_1,\omega_2)$ and $L(\mathcal{G})$ is a linear space, it contains all complements of analytic sets. 
	However, the projection of $A^c$ on the first component (denoted by $N\subset [0,1]$) needs not to be universally measurable, and $\mathcal{E}(1_{A^c}|\mathcal{H})(\omega)=1_N(\omega_1)$ implies that $L(\mathcal{H})$ contains non-universally measurable functions (in fact even non-Lebesgue-measurable ones \cite[Section 5.4]{nutz2013constructing}).
	But this implies that even if $\mathcal{E}(\cdot):=E_P[\cdot]$ for some $P$, one can not define $\mathcal{E}(X)$ for $X\in L(\mathcal{H})$.
\end{remark}

\section{Applications, extensions, and examples}
\label{sec:example.extension}

\subsection{Risk measures under Knightian uncertainty}
\label{sec:sub.risk}

Let $(\Omega,(\mathcal{F}_t)_{t=1,\dots,T},\mathcal{F})$ be a filtered space, where $\Omega$ is Polish with Borel $\sigma$-field $\mathcal{F}$, $T\in\mathbb{N}$, and each $\mathcal{F}_t$ is assumed to be countably generated.
For every $t$, let $\mathcal{E}(\cdot|\mathcal{F}_t)\colon\mathcal{L}(\mathcal{F})\to\mathcal{L}(\mathcal{F}_t)$ 
be a conditional sublinear expectation which satisfies (A) and (B) and therefore has the representation 
$\mathcal{E}(X|\mathcal{F}_t)(\omega)=\sup_{P\in\mathcal{P}_{\mathcal{F}_t}(\omega)} E_P[X]$ as in Theorem \ref{thm:cond.rep.sublin}.
By (a slight modification of) Lemma \ref{lem:composition} the composition  $\mathcal{E}_{t,T}(\cdot):=\mathcal{E}(\mathcal{E}(\cdots \mathcal{E}(\cdot|\mathcal{F}_{T})\cdots|\mathcal{F}_{t+1})|\mathcal{F}_t)$ defines a sublinear expectation from $\mathcal{L}(\mathcal{F})$ to $\mathcal{L}(\mathcal{F}_t)$ with representation
\[ \mathcal{E}_{t,T}(X)(\omega)=\sup_{P\in\mathcal{P}_{t,T}(\omega)} E_P[X]
\quad\text{for }\omega\in\Omega\text{ and }X\in\mathcal{L}(\mathcal{F}),\]
where $\mathcal{P}_{t,T}(\omega):=\mathcal{P}_{\mathcal{F}_t}(\omega)\otimes\cdots\otimes \mathcal{P}_{\mathcal{F}_{T}}$.
Now let $l\colon\mathbb{R}\to\mathbb{R}$, $x\mapsto x^+/\lambda$ for some $\lambda\in(0,1)$,
and define the time-consistent robust average value at risk
by $\mathcal{R}_{t,T}(\cdot):=\mathcal{R}(\mathcal{R}(\cdots \mathcal{R}(\cdot|\mathcal{F}_T)\cdots|\mathcal{F}_{t+1})|\mathcal{F}_t)$, where
\begin{align}
	\label{eq:avar}
	\mathcal{R}(X|\mathcal{F}_t)(\omega)
	:= \inf_{s\in\mathbb{R}} \Big(\mathcal{E}(l(X-s)|\mathcal{F}_t)(\omega) +s\Big)
\end{align}
for $\omega\in\Omega$ and $X\in\mathcal{L}(\mathcal{F})$.

\begin{example}
\label{ex:avar.robust}
	For every $t$, the functional $\mathcal{R}(\cdot|\mathcal{F}_t)$ is a conditional sublinear expectation	which satisfies condition (A) and (B) and has the representation
	\[ \mathcal{R}(X|\mathcal{F}_t)(\omega)
	=\sup_{Q\in\mathcal{Q}_{\mathcal{F}_t}(\omega)} E_Q[X]
	\quad\text{for } \omega\in\Omega \text{ and } X\in\mathcal{L}(\mathcal{F})	 \]
	where
	$\mathcal{Q}_{\mathcal{F}_t}(\omega)$ is the set of all probabilities $Q$ for which there exists $P\in\mathcal{P}_{\mathcal{F}_t}(\omega)$ such that $Q$ is absolutely continuous w.r.t.~$P$ and the Radon-Nykodim derivative $dQ/dP$ is bounded by $1/\lambda$. 
	Moreover
	\[ \mathcal{R}_{t,T}(X)(\omega)=\sup_{Q\in\mathcal{Q}_{t,T}(\omega)}E_Q[X]
	\quad\text{for } \omega\in\Omega \text{ and } X\in\mathcal{L}(\mathcal{F}), \]
	where $\mathcal{Q}_{t,T}(\omega):=\mathcal{Q}_{\mathcal{F}_t}(\omega)\otimes\cdots\otimes \mathcal{Q}_{\mathcal{F}_T}$.
\end{example}

\begin{remark}
	In a one-period setting (i.e.~$T=1$), one class of examples in robust mathematical finance emerges from taking an estimator $P^\ast$ and replacing $E_{P^\ast}[\cdot]$ by $\mathcal{E}(\cdot):=\sup_{P\in\mathcal{P}}E_P[\cdot]$, where $\mathcal{P}$ is the neighborhood (say in Wasserstein distance) of $P^\ast$; see e.g.~\cite{bartl2017computational,obloj2018statistical} and references therein for a motivation.
	It can be shown that $\mathcal{R}(\cdot|\mathcal{F}_t)$ then has a simple formula \cite[Example 2.10]{bartl2017computational}.
	This example has a natural lift to a multi-period setting by considering the set $\mathcal{P}$ of probabilities for which all conditional distributions (w.r.t.~the filtration) are in the neighborhood of the conditional distributions $P^\ast$; see e.g.~\cite[Chapter 2.3]{bartl2016exponential} for a discussion.
	When applying standard static methods to tackle optimization problems with $E_{P^\ast}[\cdot]$ replaced by $\mathcal{E}(\cdot)$, it is often crucial for the set $\mathcal{P}$ to be compact.
	In Example \ref{ex:wasserstein.compact} below, it will be shown that a Feller condition on $P^\ast$ is sufficient to guarantee compactness of $\mathcal{P}$.
\end{remark}

\begin{proof}[{\bf Proof of Example \ref{ex:avar.robust}}]
	As $l$ is increasing, one has $l(X-s)\in\mathcal{L}(\mathcal{F})$ for every $X\in\mathcal{L}(\mathcal{F})$ and $s\in\mathbb{R}$ so that $\mathcal{R}(\cdot|\mathcal{F}_t)$  is well-defined.
	Moreover, as the mapping $s\mapsto \mathcal{E}(l(X-s)|\mathcal{F}_t)(\omega) +s$ is convex and real-valued for every $\omega\in\Omega$, it is continuous and one may restrict the infimum in \eqref{eq:avar} to $s\in\mathbb{Q}$. 
	Therefore, as the countable infimum, $\mathcal{R}(X|\mathcal{F}_t)$ is $\mathcal{F}_t$-upper semianalytic whenever $X$ is.
	Elementary computations show that $\mathcal{R}(\cdot|\mathcal{F}_t)$ is increasing, sublinear, and satisfies $\mathcal{R}(X|\mathcal{F}_t)=X$ for $X\in\mathcal{L}(\mathcal{F}_t)$.
	By interchanging two infima and the fact that $\mathcal{E}(\cdot|\mathcal{F}_t)$ satisfies (A), one gets that $\mathcal{R}(\cdot|\mathcal{F}_t)$ satisfies (A) as well.
	As for condition (B), fix $\omega\in\Omega$ and let $X_n\in\mathcal{L}(\mathcal{F})$ be a sequence which increases pointwise to $X\in\mathcal{L}(\mathcal{F})$. 
	For every $n$, let $s_n\in\mathbb{R}$ such that $\mathcal{E}(l(X_n-s_n)|\mathcal{F}_t)(\omega) +s_n \leq \mathcal{R}(X_n|\mathcal{F}_t)(\omega)+1/n$.
	As the sequence $X_n$ is bounded uniformly in $n$ and $\lambda\in(0,1)$, it follows that $s_n$ is bounded and thus	a subsequence, still denoted by $s_n$, converges. 
	The dual representation of $\mathcal{E}(\cdot|\mathcal{F}_t)$ now implies that
	\[
	\mathcal{R}(X|\mathcal{F}_t)(\omega)
	\leq \mathcal{E}(l(X-s)|\mathcal{F}_t)(\omega) +s
	\leq \liminf_n\Big( \mathcal{E}(l(X_n-s_n)|\mathcal{F}_t)(\omega) +s_n\Big).\]
	The last term is smaller than $\mathcal{R}(X|\mathcal{F}_t)(\omega)$ due to the choice of $s_n$ and the fact that $\mathcal{R}(X_n|\mathcal{F}_t)\leq\mathcal{R}(X|\mathcal{F}_t)$ for each $n$.
	Hence $\mathcal{R}(\cdot|\mathcal{F}_t)$ satisfies condition (B).
	The specific form of $\mathcal{Q}_{\mathcal{F}_t}$ follows as in the case without Knightian uncertainty
	\cite[Lemma 4.51 and Theorem 4.52]{follmer2011stochastic}, additionally using a suitable minimax theorem \cite[Theorem 2]{fan1953minimax}, see \cite[Theorem 3.4]{bartl2017computational}.
	The representation of $\mathcal{R}_{t,T}(\cdot)$ is due to (a slight modification of) Lemma \ref{lem:composition}.
\end{proof}

The nonrobust version of Example \ref{ex:avar.robust} can be found e.g.~\cite[Example 2.3.1]{cheridito2011composition}.
For further literature which uses dynamic programming and conditional nonlinear expectations in the context of mathematical finance under Knightian uncertainty in discrete time, see e.g.~\cite{aksamit2016robust,carassus2016robust,bouchard2016super,burzoni2016pointwise,neufeld2016robust,nutz2014utility}.

\subsection{Compactness for product of measures and kernels}
\label{sec:compactness.kernel}

\begin{proof}[{\bf Proof of Theorem \ref{thm:compactness.of.product}}]

	Assume first that $\omega\mapsto \max_{P\in\mathcal{P}_\mathcal{G}(\omega)} E_P[X]$ is upper semicontinuous	for every $X\in C_b(\Omega)$ and that the values of $\mathcal{P}_\mathcal{G}$ are compact.
	Define the functional
	\[\mathcal{E}(\cdot|\mathcal{G})\colon\mathcal{L}(\mathcal{F})\to\mathbb{R}^\Omega,
	\quad \mathcal{E}(X|\mathcal{G})(\omega):=\sup_{P\in\mathcal{P}_\mathcal{G}(\omega)} E_P[X].\]
	Then, by convexity and compactness of each $\mathcal{P}_\mathcal{G}(\omega)$, the hyperplane separation theorem shows that
	\[\mathop{\mathrm{Graph}}(\mathcal{P}_\mathcal{G})
	=\{ (\omega,P)\in \Omega\times\mathfrak{P}(\Omega) : E_P[X]\leq \mathcal{E}(X|\mathcal{G})(\omega) 
	\text{ for all }X\in C_b(\Omega)  \}. \]
	Further, the same argumentation as in the proof of Theorem \ref{thm:cond.rep} shows that one can restrict to all $X$ in a countable set $D\subset C_b(\Omega)$.
	For every $X\in D$ the mapping $(\omega,P)\to E_P[X]-\mathcal{E}(X|\mathcal{G})(\omega)$ is $\mathcal{G}\otimes\mathcal{B}(\mathfrak{P}(\Omega))$-measurable by assumption so that $\mathop{\mathrm{Graph}}(\mathcal{P}_\mathcal{G})$, as a countable intersection, has the same measurability.
	Now Theorem \ref{thm:cond.rep.sublin} implies that $\mathcal{E}(\cdot|\mathcal{G})$ is a conditional nonlinear expectation which satisfies (A) and (B); the same holds true for $\mathcal{E}'(\cdot)\colon\mathcal{L}(\mathcal{F})\to\mathbb{R}$ defined by $\mathcal{E}'(X):=\sup_{P\in\mathcal{P}}E_P[X]$.
	Therefore $\mathcal{E}(\cdot):=\mathcal{E}'(\mathcal{E}(\cdot|\mathcal{G}))$ defines a nonlinear expectation, which clearly satisfies (B).
	Moreover, it also satisfies (A).
	Indeed, let $X_n\in C_b(\Omega)$ be a sequence which decreases pointwise to $X\in\mathcal{L}(\mathcal{F})$.
	By assumption $\mathcal{E}(X_n|\mathcal{G})\in usc_b(\Omega)$ decreases pointwise to $\mathcal{E}(X|\mathcal{G})$, therefore $\mathcal{E}(X_n)$ decreases to $\mathcal{E}(X)$ by Remark \ref{rem:cont.above.usc}.
	Now Theorem \ref{thm:cond.rep.sublin} implies that $\mathcal{E}(X)=\sup_{P\in\mathcal{Q}} E_P[X]$ for a (by Remark \ref{rem:unique.penalty} unique) convex compact set $\mathcal{Q}\subset\mathfrak{P}(\Omega)$.
	As $\mathcal{E}(\cdot)=\mathcal{E}(\mathcal{E}(\cdot|\mathcal{G}))$ by definition,	Theorem \ref{thm:char.tower.sublin} yields that
	$\mathcal{Q}=\mathcal{P}\otimes\mathcal{P}_\mathcal{G}$, which proves the claim.

	To show the reverse direction, assume that $\mathcal{P}\otimes\mathcal{P}_\mathcal{G}$ is compact for every compact convex set $\mathcal{P}$. 	
	If $\mathcal{P}_\mathcal{G}(\omega)$ is not compact for some $\omega\in\Omega$, then neither is $\mathcal{P}\otimes\mathcal{P}_\mathcal{G}$ for $\mathcal{P}:=\{\delta_\omega\}$.
	So assume that $\mathcal{P}_\mathcal{G}$ has compact values but $\omega\mapsto \max_{P\in\mathcal{P}_\mathcal{G}(\omega)} E_P[X]$ is not upper semicontinuous for some $X\in C_b(\Omega)$, i.e.~there is $\omega\in\Omega$ and a sequence $\omega_n\in \Omega$ converging to $\omega$ such that 
	\begin{align}
	\label{eq:compact.limsup}
	\limsup_n \max_{P\in\mathcal{P}_\mathcal{G}(\omega_n)} E_P[X] > \max_{P\in\mathcal{P}_\mathcal{G}(\omega)} E_P[X].
	\end{align}
	For every $n$, pick some $P_n\in\mathcal{P}_\mathcal{G}(\omega_n)$ which attains the maximum in the
	left hand side of \eqref{eq:compact.limsup}. 
	After passing to a subsequence (still denoted by $P_n$), one may assume that $E_{P_n}[X]$ converges to the left hand side of \eqref{eq:compact.limsup}.
	As $C:=\{\omega_n : n\in\mathbb{N}\}\cup\{\omega\} \subset \Omega$ is compact, the set $\mathcal{P}:=\{P \in\mathfrak{P}(\Omega) : P(C)=1\}$ is also compact (and obviously convex).
	Now distinguish between two cases.
	If $P_n$ does not have any convergent subsequence, then neither does $\delta_{\omega_n}\otimes P_n\in\mathcal{P}\otimes\mathcal{P}_\mathcal{G}$	which implies that the latter set cannot be compact.
	Otherwise, possibly after passing to a subsequence, $P_n$ converges to some $P$ and one has $P\notin\mathcal{P}_\mathcal{G}(\omega)$ by \eqref{eq:compact.limsup}.
	However, as
	\[ \mathcal{P}\otimes\mathcal{P}_\mathcal{G}\ni \delta_{\omega_n}\otimes P_n
	\to \delta_{\omega}\otimes P\notin\mathcal{P}\otimes\mathcal{P}_\mathcal{G},\]
	this implies that $\mathcal{P}\otimes\mathcal{P}_\mathcal{G}$ is not closed and completes the proof.
\end{proof}

\begin{remark}
\label{rem:conditions.compactness.kernels}
	By a variant of Berge's maximum theorem \cite[Lemma 17.30]{aliprantis2006infinite}, if $\mathcal{P}_\mathcal{G}\colon \Omega\rightsquigarrow\mathfrak{P}(\Omega)$ is upper hemi-continuous with nonempty compact values, the mapping $\omega\mapsto \max_{P\in\mathcal{P}_\mathcal{G}(\omega)} E_P[X]$ is upper semicontinuous for $X\in C_b(\Omega)$.

	Let $\phi$ in Theorem \ref{thm:compactness.of.product} such that $\phi(\Omega)$ is Borel and there exists $\psi\colon \phi(\Omega)\to\Omega$ Borel with $\psi(s)\in \phi^{-1}(\{s\})$ for all $s\in\phi(\Omega)$.
	Then $\omega\mapsto \max_{P\in\mathcal{P}_\mathcal{G}(\omega)} E_P[X]$ is $\mathcal{G}$-measurable if and only if it is $\mathcal{F}$-measurable and $\mathcal{G}$-indistinguishable, that is, $\phi(\omega)=\phi(\eta)$ implies $\max_{P\in\mathcal{P}_\mathcal{G}(\omega)} E_P[X]=\max_{P\in\mathcal{P}_\mathcal{G}(\eta)} E_P[X]$ for all $\omega,\eta\in\Omega$.	
\end{remark}

\begin{example}
\label{ex:wasserstein.compact}
	Let $\Omega=\mathbb{R}^d$ with euclidean distance, let $p\in[1,\infty)$, and denote by $\mathcal{W}_p$ the $p$-Wasserstein distance; see e.g.~\cite[Chapter 6]{villani2008optimal}.
	Let $K\colon\Omega\to\mathfrak{P}(\Omega)$ be $\mathcal{G}$-measurable and continuous w.r.t.~$\mathcal{W}_p$ such that $E_{K(\omega)}[|\mathrm{id}|^p]<+\infty$ for all $\omega\in\Omega$, and define
	\[\mathcal{P}_\mathcal{G}(\omega):=\{ P\in\mathfrak{P}(\Omega) : \mathcal{W}_p(P,K(\omega))\leq\delta \text{ and } P([\omega]_\mathcal{G})=1 \},\]
	where $\delta>0$ is fixed.
	If $\mathcal{G}=\sigma(\phi)$ for some continuous $\phi\colon\Omega\to S$ as in Remark \ref{rem:conditions.compactness.kernels}, then $\omega\mapsto \sup_{P\in\mathcal{P}_\mathcal{G}(\omega)} E_P[X]$ is upper semicontinuous and $\mathcal{G}$-measurable for $X\in C_b(\Omega)$.
	In particular, for every compact set $\mathcal{P}\subset\mathfrak{P}(\Omega)$, by Theorem \ref{thm:compactness.of.product}, the set $\mathcal{P}\otimes\mathcal{P}_\mathcal{G}$ is compact.

	In case that $\mathcal{P}$ itself is the neighborhood of some measure $P$ with $E_P[|\mathrm{id}|^p]<+\infty$ and $E_{K(\omega)}[|\mathrm{id}|^p]\leq c(1+|\omega|^p)$ for some constant $c$, the set $\mathcal{P}\otimes\mathcal{P}_\mathcal{G}$ is compact under the topology induced by $\mathcal{W}_q$ for every $q\in[1,p)$.
\end{example}
\begin{proof}
	With the notion of Remark \ref{rem:conditions.compactness.kernels}, $\mathcal{P}_\mathcal{G}$ is $\mathcal{G}$-indistinguishable.
	Therefore, by the mentioned remark, upper semicontinuity of $\omega\mapsto \sup_{P\in\mathcal{P}_\mathcal{G}(\omega)} E_P[X]$ for $X\in C_b(\Omega)$ also implies $\mathcal{G}$-measurability of that mapping.
	To show upper semicontinuity, let $\omega_n\in\Omega$ be a sequence with $\omega_n\to\omega\in\Omega$.
	As $K$ is continuous, the set $\{K(\omega_n):n\in\mathbb{N}\}\cup\{K(\omega)\}$ is compact.
	This can be used to show that $\bigcup \{ \mathcal{P}_\mathcal{G}(\omega_n) : n\in\mathbb{N}\}\cup \mathcal{P}_\mathcal{G}(\omega)$ is relatively compact.
	Thus, if $P_n$ denotes a (near) maximizer for $\sup_{P\in\mathcal{P}_\mathcal{G}(\omega_n)} E_P[X]$ for each $n$,
	there exists $P$ and a subsequence, still denoted by $P_n$, which converges to $P$.
	Now, by lower semicontinuous of $\mathcal{W}_p$ (which follows e.g.~from the dual representation \cite[Theorem 5.9]{villani2008optimal}), it follows that 
	\[\mathcal{W}_p(P,K(\omega))
	\leq\liminf_n \Big( \mathcal{W}_p(P_n,K(\omega_n)) + \mathcal{W}_p(K(\omega_n),K(\omega)) \Big)
	\leq\delta.\]
	Moreover, for every $m\in\mathbb{N}$, the set $C_m:=\bigcup \{ [\omega_n]_\mathcal{G} : n\geq m\}\cup[\omega]_\mathcal{G}=\{ \eta\in\Omega : \phi(\eta)=\phi(\theta) \text{ for some } \theta\in\{\omega,\omega_m,\omega_{m+1},\dots\}\}$ is closed and $P_n(C_m)=1$ for all $n\geq m$.
	Therefore $P(C_m)=1$ for all $m$, hence $P([\omega]_\mathcal{G})=1$ and so $P\in\mathcal{P}_\mathcal{G}(\omega)$.
	This shows the desired upper semicontinuity.
	The second statement is a consequence of a characterization of the $\mathcal{W}_p$ topology \cite[Theorem 6.8]{villani2008optimal} and a Della-Valle-Poussin type result for probabilities \cite[Corollary A.47.]{follmer2011stochastic}.
\end{proof}

\subsection{Fubini's theorem}

Let $\Omega=\Omega_1\times\Omega_2$ be the product of two Polish spaces.
While Lemma \ref{lem:composition} can be seen as a nonlinear version of Fubini's theorem on the existence of the product of a measure and a kernel, one can also ask if there is a nonlinear version of Fubini's classical theorem (i.e.~on the possibility to interchange the order of integration when two measures are replaced by two sets of measures).
In general this is not true any more (take for example $\Omega_1:=\Omega_2:=[0,1]$ as well as $\mathcal{P}_1:=\{(\delta_0+\delta_1)/2\}$ and $\mathcal{P}_2:=\mathfrak{P}([0,1])$) but it is possible to characterize when interchanging the order is possible.

\begin{proposition}
\label{prop:fubini}
	Let $\mathcal{P}_1\subset\mathfrak{P}(\Omega_1)$ and $\mathcal{P}_2\subset\mathfrak{P}(\Omega_2)$ be two convex and compact sets of probabilities. Then it holds
	\begin{align*} 
	&\sup_{P\in\mathcal{P}_1} \int_{\Omega_1} \sup_{Q\in\mathcal{P}_2}\int_{\Omega_2} X(\omega_1,\omega_2)\,Q(d\omega_2) P(d\omega_1)\\
	&=\sup_{Q\in\mathcal{P}_2}\int_{\Omega_2} \sup_{P\in\mathcal{P}_1} \int_{\Omega_1} X(\omega_1,\omega_2)\,P(d\omega_1) Q(d\omega_2)
	\end{align*}
	for all $X\in\mathcal{L}(\Omega)$ if and only if
	\begin{align*} 
	&\{ P\otimes R: P\in\mathcal{P}_1 \text{ and }R\colon \Omega_1\to\mathfrak{P}(\Omega_2) \text{ kernel with }
	R(\cdot)\in\mathcal{P}_2\, P\text{-as}\}\\
	&=\{ (Q\otimes R)\circ\pi^{-1}: Q\in\mathcal{P}_2 \text{ and } R\colon \Omega_2\to\mathfrak{P}(\Omega_1) \text{ kernel with }
	R(\cdot)\in\mathcal{P}_1\, Q\text{-as}\},
	\end{align*}
	where $\pi\colon \Omega_1\times\Omega_2 \to \Omega_2\times\Omega_1$ is given by $\pi(\omega_1,\omega_2):=(\omega_2,\omega_1)$.
\end{proposition}
\begin{proof}
	The proof is similar to the one given for Theorem \ref{thm:compactness.of.product} but somewhat	notationally involved, and shall be skipped.
\end{proof}

\subsection{Controlled Brownian motion}

This last example is in the spirit of \cite{nutz2013constructing,peng2008multi}; it's purpose is to illustrate the results in a continuous time setting.
For some fixed time horizon $T>0$ let $\Omega:=C([0,T],\mathbb{R})$ endowed with raw filtration $(\mathcal{F}_t)_{t\in[0,T]}$ and a stopping time $\tau$.
Let $B$ be the canonical process on $\Omega$, denote by $W$ the Wiener measures, fix two numbers $0<\underline{\sigma}<\overline{\sigma}$, and write $\Sigma$ for the set of all progressively measurable processes $\sigma\colon[0,T]\times C([0,T])\to\mathbb{R}$ to $B$ which satisfy $\sigma\in[\underline{\sigma}^2,\overline{\sigma}^2]$ $W\times dt$-almost surely.
For $\sigma\in\Sigma$, denote by $B^{\omega,t,\sigma}:=\omega 1_{[0,t]}+(\omega(t)+\int_t^\cdot \sigma_s \,dB_s)1_{(t,T]}$ the Brownian motion with volatility $\sigma$ starting in $(t,\omega(t))$.

\begin{example}
\label{ex:G.brownian}
	The functional
	\[ \mathcal{E}(X|\mathcal{F}_\tau)(\omega):=\sup_{\sigma\in\Sigma} E_W[X(B^{\omega,\tau(\omega),\sigma})] 
	\quad\text{for } \omega\in \Omega \text{ and } X\in\mathcal{L}(\mathcal{F}).\]
	defines a conditional sublinear expectation.	
\end{example}
\begin{proof}
	First note that by Galmarino's test, $\mathcal{F}_\tau$ is countably generated.
	Moreover, Galmarino's test extends to all $\mathcal{F}$-upper semianalytic functions.
	Therefore $\mathcal{E}(X|\mathcal{F}_\tau)=X$ for every $X\in\mathcal{L}(\mathcal{F}_\tau)$.
	To show that $\mathcal{E}(\cdot|\mathcal{F}_\tau)$ maps into $\mathcal{L}(\mathcal{F}_\tau)$, endow $\Sigma$ with the norm $\|\sigma\|:=E_W[\int_0^T\sigma_s^2 \,ds]^{1/2}$ which renders $\Sigma$ a Polish space.
	Then the mapping
	\[\Omega\times \Sigma \ni (\omega,\sigma)\mapsto 
	P^{\omega,\tau(\omega),\sigma}:=W\circ (B^{\omega,\tau(\omega),\sigma})^{-1} \in \mathfrak{P}(\Omega)\]
	is $\mathcal{F}\otimes\mathcal{B}(\Sigma)$-measurable.
	Indeed, for Lipschitz-continuous $X\colon\Omega\to\mathbb{R}$, this is a consequence of Doob's inequality, and a monotone class argument yields that this carries over to all bounded $\mathcal{F}$-measurable $X\colon\Omega\to\mathbb{R}$.
	Thus $(\omega,\sigma)\mapsto P^{\omega,\tau(\omega),\sigma}$ is $\mathcal{F}\otimes\mathcal{B}(\Sigma)$-measurable \cite[Proposition 7.26]{bertsekas1978stochastic}, hence for every $X\in\mathcal{L}(\mathcal{F})$
	\[(\omega,\sigma)\to E_{P^{\omega,\tau(\omega),\sigma}}[X]=E_W[X(B^{\omega,\tau(\omega),\sigma})]\]
	is $\mathcal{F}\otimes\mathcal{B}(\Sigma)$-upper semianalytic \cite[Proposition 7.48]{bertsekas1978stochastic}.
	Therefore $\mathcal{E}(X|\mathcal{F}_\tau)$ is $\mathcal{F}$-upper semianalytic \cite[Corollary 6.10.10]{bogachev2007measure}, and Gamarino's test implies that it is actually $\mathcal{F}_\tau$-upper semianalytic.
\end{proof}

\appendix

\section{Analytic sets}
\label{sec:appendix.analytic}

Let $(S,\mathcal{A})$ be a measurable space.
Then a subset $A\subset S$ is call $\mathcal{A}$-analytic if it is the nucleus of a Suslin-scheme, that is, there are sets $A_{n_1,\dots,n_k}\in\mathcal{A}$ for every $k$ and $(n_1,\dots,n_k)\in\mathbb{N}^k$ such that $A=\bigcup_{(n_k)\in\mathbb{N}^\mathbb{N}}\bigcap_{k\in\mathbb{N}} A_{n_1,\dots,n_k}$.
The set of all $\mathcal{A}$-analytic sets is stable under countable intersections and unions, however, not under complementation. 
Another (useful) representation of $\mathcal{A}$-analytic sets is through the projection of higher dimensional sets, see e.g.~\cite[Chapter 6.10(ii)]{bogachev2007measure}.
A function $f\colon S\to[-\infty,+\infty]$ is call $\mathcal{A}$-upper (resp.~lower) semianalytic, if $\{f\geq c\}$ (resp.~$\{f\leq c\}$) is a $\mathcal{A}$-analytic set for every $c\in\mathbb{R}$.
If $S$ is a Polish space together with Borel $\sigma$-field $\mathcal{A}=\mathcal{B}(\mathcal{S})$, it follows from the definition of $\mathcal{B}(S)$-analytic sets that every Borel set is $\mathcal{B}$-analytic, and from Lusin's theorem \cite[Proposition 7.42]{bertsekas1978stochastic} that every $\mathcal{B}(S)$-analytic set is universally measurable.
The same of course holds true if sets are replaced by functions in the previous sentence.
For a countable family $\{X_n: n\in\mathbb{N}\}$ of $\mathcal{A}$-upper semianalytic functions, $X_1+X_2$, $\sup_n X_n$, and $\inf_n X_n$ are again $\mathcal{A}$-upper semianalytic.
A comprehensive treatment of analytic sets well suited for the present setting can be found in \cite[Chapter 7]{bertsekas1978stochastic} or \cite{bogachev2007measure}.	

\section{Choquet's theorem}
\label{sec:app.choquet}

For convenience, a brief sketch the proof of Choquet's theorem for nonlinear expectations is given below;
a detailed proof is given e.g.~in \cite[Section 2]{bartl2017robust}.
Let $\mathcal{E}(\cdot)\colon\mathcal{L}(\mathcal{F})\to\mathbb{R}$ be an increasing convex functional which preserves the constants, is continuous from above on $C_b(\Omega)$, and continuous from below on $\mathcal{L}(\mathcal{F})$.
By monotonicity, $\mathcal{E}(\cdot)$ is continuous w.r.t.~the maximum norm on $C_b(\Omega)$ so that the Fenchel-Moreau / Hahn-Banach theorem implies
\[ \mathcal{E}(X)=\max_{P\in C_b(\Omega)^\ast} (\langle X,P \rangle -\mathcal{E}^\ast(P)) \quad\text{for } X\in C_b(\Omega), \]
where $\mathcal{E}^\ast(P):=\sup_{X\in C_b(\Omega)} (\langle X,P \rangle - \mathcal{E}(X))$ and $C_b(\Omega)^\ast$ denotes the topological dual of $(C_b(\Omega),\|\cdot\|_\infty)$.
For every $P$ with $\mathcal{E}^\ast(P)<+\infty$ a scaling argument implies that $P$ needs to be an increasing functional satisfying $\langle 1,P\rangle=1$ (indeed, if for example $\langle 1,P\rangle\neq 1$, then $\mathcal{E}^\ast(P)\geq\sup_{\lambda\in\mathbb{R}}(\langle \lambda,P\rangle -\lambda)=+\infty$).
Moreover, as $\mathcal{E}(\cdot)$ is continuous from above on $C_b(\Omega)$, one can show that every $P$ with $\mathcal{E}^\ast(P)<+\infty$ has this property as well.
Therefore, by the Daniell-Stone theorem, $P$ can be viewed as a probability on $\sigma(C_b(\Omega))=\mathcal{F}$ and $\langle X,P\rangle=E_P[X]$ for all $X\in C_b(\Omega)$.
This implies that
\[ \mathcal{E}(X)=\max_{P\in\mathfrak{P}(\Omega)} (E_P[X] -\mathcal{E}^\ast(P)) \quad\text{for } X\in C_b(\Omega) \]
and, using a minimax theorem (similar as in the proof of Theorem \ref{thm:cond.rep}), this equality extends to $X\in usc_b(\Omega)$.
In fact, as a closed subset of the unit sphere, $\{\mathcal{E}^\ast\leq c\}$ is compact in $\sigma(C_b(\Omega)^\ast,C_b(\Omega))$  by the Banach-Alaoglu theorem. 
As one can further show that $\{\mathcal{E}^\ast\leq c\}$ is uniformly continuous from above, using the Daniell-Stone theorem once more, one obtains that $\{\mathcal{E}^\ast\leq c\}$, as a subset of $\mathfrak{P}(\Omega)$, is weakly compact.
Now notice that $\mathcal{E}(\cdot)$ is a (functional) capacity in the sense of Choquet by assumption, therefore his regularity result yields
\[ \mathcal{E}(X)=\sup_{Y\leq X, \,Y\in usc_b(\Omega)} \mathcal{E}(Y)
= \sup_{Y\leq X, \,Y\in usc_b(\Omega)} \max_{P\in\mathfrak{P}(\Omega)} (E_P[Y] -\mathcal{E}^\ast(P))\] 
for every function $X$ which can be written as the nucleus of a Suslin scheme in $C_b(\Omega)$, see \cite[Section 3]{choquet1959forme}, that is, for every $X\in\mathcal{L}(\mathcal{F})$.
The representation $\mathcal{E}(X)=\sup_{P\in\mathfrak{P}(\Omega)} (E_P[X] -\mathcal{E}^\ast(P))$ for $X\in\mathcal{L}(\mathcal{F})$ now follows from the representation of $\mathcal{E}(Y)$ for $Y\in usc_b(\Omega)$, interchanging two suprema, and the fact that $\sup_{Y\leq X, \,Y\in usc_b(\Omega)} E_P[Y]=E_P[X]$ (to see this apply for example Choquet's results to $E_P[\cdot]$).
Finally, in case that $\mathcal{E}(\cdot)$ is sublinear, it follows from a scaling argument that as in the proof of Theorem \ref{thm:cond.rep} that $\mathcal{E}^\ast$ only takes the values $0$ and $+\infty$.
To recover the stated in the introduction it therefore remains to set $\mathcal{P}:=\{ P\in\mathfrak{P}(\Omega) : \mathcal{E}^\ast(P)=0\}$.

\vspace{1.5em}
{\bf Acknowledgment:}
The author would like to thank Samuel Drapeau, Hans F\"ollmer, Michael Kupper, and two anonymous referees for fruitful discussions and helpful suggestions.
The author has been funded by the Vienna Science and Technology Fund (WWTF) through project VRG17-005 and by the Austrian Science Fund (FWF) under grant Y00782.

\bibliographystyle{abbrv}

\begin{thebibliography}{10}
	\bibitem{aksamit2016robust}
	A.~Aksamit, S.~Deng, J.~Ob{\l}{\'o}j, and X.~Tan.
	\newblock Robust pricing--hedging duality for American options in discrete time financial markets. 
	\newblock {\em Mathematical Finance}, forthcoming (arXiv:1604.05517).
	
	\bibitem{aliprantis2006infinite}
	C.~Aliprantis and K.~Border.
	\newblock {\em Infinite Dimensional Analysis: A Hitchhiker's Guide}.
	\newblock Springer Science \& Business Media, 2006.

	\bibitem{bartl2016exponential}
	D.~Bartl.
	\newblock Exponential utility maximization under model uncertainty for
	unbounded endowments.
	\newblock {\em The Annals of Applied Probability}, 29(1), 577--612, 2019.

	\bibitem{bartl2017computational}
	D.~Bartl, S.~Drapeau, and L.~Tangpi.
  	\newblock Computational aspects of robust optimized cer-tainty equivalents and option pricing.
  	\newblock {\em Mathematical Finance}, forthcoming (arXiv:1706.10186).

	\bibitem{bartl2017robust}
	D.~Bartl, P.~Cheridito, and M.~Kupper.
	\newblock Robust expected utility maximization with medial limits.
	\newblock {\em Journal of Mathematical Analysis and Applications}, 471(1-2), 752--775, 2019.
	
	\bibitem{beiglbock2015complete}
	M.~Beiglb{\"o}ck, M.~Nutz, and N.~Touzi.
	\newblock Complete duality for martingale optimal transport on the line.
	\newblock {\em The Annals of Probability}, 45(5):3038--3074, 2017.
		
	\bibitem{bertsekas1978stochastic}
	D.~Bertsekas and S.~Shreve.
	\newblock {\em Stochastic optimal control: The discrete time case}, volume~23.
	\newblock Academic Press New York, 1978.

	\bibitem{bogachev2007measure}
  	V.~Bogachev.
  	\newblock {\em Measure theory, Volume II}
	\newblock Springer Science and Business Media, 2007.

	\bibitem{carassus2016robust}
	R.~Blanchard and L.~Carassus.
	\newblock Robust optimal investment in discrete time for unbounded utility
	function.
	\newblock {\em The Annals of Applied Probability}, 28(3):1856--1892, 2018.

	\bibitem{bouchard2015arbitrage}
	B.~Bouchard and M.~Nutz.
	\newblock Arbitrage and duality in nondominated discrete-time models.
	\newblock {\em The Annals of Applied Probability}, 25(2):823--859, 2015.

	\bibitem{bouchard2016super}
	B.~Bouchard, S.~Deng, and X.~Tan. 
	\newblock Super-replication with proportional transaction cost under model uncertainty.
	\newblock {\em Mathematical Finance}, forthcoming (arXiv:1707.09158).

	\bibitem{burzoni2016pointwise}	
	M.~Burzoni, M.~Frittelli, Z.~Hou, M.~Maggis, and J.~Ob{\l}{\'o}j.
	\newblock Pointwise Arbitrage pricing theory in discrete time.
	\newblock {\em Mathematics of Operations Research}, forthcoming (arXiv:1612.07618).
	
	\bibitem{cheridito2011composition}
	P.~Cheridito and M.~Kupper.
	\newblock Composition of time-consistent dynamic monetary risk measures in
	discrete time.
	\newblock {\em International Journal of Theoretical and Applied Finance},
	14(01):137--162, 2011.
	
	\bibitem{choquet1959forme}
	G.~Choquet.
	\newblock Forme abstraite du th{\'e}or{\`e}me de capacitabilit{\'e}.
	\newblock In {\em Annales de l'institut Fourier}, volume~9, pages 83--89, 1959.
	
	\bibitem{cohen2012quasi}
	S.~Cohen.
	\newblock Quasi-sure analysis, aggregation and dual representations of
	sublinear expectations in general spaces.
	\newblock {\em Electron. J. Probab}, 17(62):1--15, 2012.
	
	\bibitem{delbaen2002coherent}
	F.~Delbaen.
	\newblock Coherent risk measures on general probability spaces.
	\newblock In {\em Advances in finance and stochastics}, pages 1--37. Springer,
	2002.
		
	\bibitem{dellacherie2011probabilities}
	C.~Dellacherie and P.~Meyer.
	\newblock {\em Probabilities and Potential, C: Potential Theory for Discrete
		and Continuous Semigroups}, volume 151.
	\newblock Elsevier, 2011.

	\bibitem{denis2011function}
 	L.~Denis, M.~Hu, and S.~Peng.
 	\newblock Function spaces and capacity related to a sublinear expectation: application to G-Brownian motion paths.
 	\newblock {\em Potential Analysis}, 32(2):139--161, 2011.

	\bibitem{denk2017semigroup}
	R.~Denk, M.~Kupper, and M.~Nendel.
	\newblock A semigroup approach to nonlinear L{\'e}vy processes.
	\newblock {\em arXiv preprint arXiv:1710.08130}, 2017.

	\bibitem{eckstein2017extended}
	S.~Eckstein.
	\newblock Extended laplace principle for empirical measures of a Markov chain.
	\newblock {\em Advances in Applied Probability}, forthcoming (arXiv:1709.02278).

	\bibitem{karoui2013capacities}
  	N.~El Karoui and X.~Tan.
  	\newblock Capacities, measurable selection and dynamic programming Part I: abstract framework.
  	\newblock {\em arXiv preprint arXiv:1310.3363}, 2013.

	\bibitem{fan1953minimax}
	K.~Fan.
	\newblock Minimax theorems.
	\newblock {\em Proceedings of the National Academy of Sciences}, 39(1):42--47,
	1953.
	
	\bibitem{follmer2011stochastic}
	H.~F{\"o}llmer and A.~Schied.
	\newblock {\em Stochastic Finance: An Introduction in Discrete Time}.
	\newblock Walter de Gruyter, 2011.
		
	\bibitem{katvetov1951real}
	M.~Kat{\v{e}}tov.
	\newblock On real-valued functions in topological spaces.
	\newblock {\em Fundamenta Mathematicae}, 38(1):85--91, 1951.
	
	\bibitem{kellerer1984duality}
	H.~Kellerer.
	\newblock Duality theorems for marginal problems.
	\newblock {\em Zeitschrift f{\"u}r Wahrscheinlichkeitstheorie und verwandte
		Gebiete}, 67(4):399--432, 1984.
	
	\bibitem{lacker2016non}
	D.~Lacker.
	\newblock A non-exponential extension of sanov's theorem via convex duality.
	\newblock {\em arXiv preprint arXiv:1609.04744}, 2016.
	
	\bibitem{maggis2016fatou}
	M.~Maggis, T.~Meyer-Brandis, and G.~Svindland.
	\newblock The fatou property under model uncertainty.
	\newblock {\em Positivity}, 22(5):1325--1343, 2018.
	
	\bibitem{neufeld2016robust}
	A.~Neufeld and M.~\v{S}iki\'{c}.
	\newblock Robust utility maximization in discrete-time markets with friction.
	\newblock {\em SIAM Journal on Control and Optimization}, 56(3):1912--1937, 2018.

	\bibitem{neufeld2017nonlinear}
	A.~Neufeld and M.~Nutz.
  	\newblock Nonlinear L{\'e}vy processes and their characteristics.
 	\newblock {\em Transactions of the American Mathematical Society}, 
 	369(1):69--95, 2017.
	
	\bibitem{nutz2013constructing}
	M.~Nutz and R.~van Handel.
	\newblock Constructing sublinear expectations on path space.
	\newblock {\em Stochastic Processes and their Applications}, 123(8):3100--3121,
	2013.

	\bibitem{nutz2014utility}
	M.~Nutz.
	\newblock Utility maximization under model uncertainty in discrete time.
	\newblock {\em Mathematical Finance}, 26(2):252--268, 2016.

	\bibitem{obloj2018statistical}
  	J.~Ob{\l}{\'o}j and J.~Wiesel.
 	\newblock Statistical estimation of superhedging prices.
  	\newblock {\em arXiv preprint arXiv:1807.04211}, 2019.

	\bibitem{peng2008multi}
  	S.~Peng.
	\newblock Multi-dimensional G-Brownian motion and related stochastic calculus under G-expectation.
	\newblock {\em Stochastic Processes and their Applications}, 118(12):2223--2253, 2008.

	\bibitem{peng2010nonlinear}
  	S.~Peng.
 	\newblock Nonlinear expectations and stochastic calculus under uncertainty.
  	\newblock {\em arXiv preprint arXiv:1002.4546}, 2010.
	
	\bibitem{stroock2010probability}
	D.~Stroock.
	\newblock {\em Probability Theory: An Analytic View}.
	\newblock Cambridge university press, 2010.

	\bibitem{villani2008optimal}
	C.~Villani.
 	\newblock {\em Optimal transport: old and new}.
  	\newblock Springer Science and Business Media, 2008.
 	
\end{thebibliography}

\end{document}